\documentclass[a4paper,11pt]{amsart}
\usepackage{amsmath,amsthm,amssymb,amsfonts,enumerate,color,esint,bbm}
\usepackage[pdftex]{graphicx}
\usepackage{float}
\usepackage{tikz}
\usepackage[shortlabels]{enumitem}
\usepackage{cancel}

\usepackage[colorlinks=true, allcolors=blue]{hyperref}
\usepackage{amsrefs}
\usepackage{pgfplots}
\pgfplotsset{
compat=newest,
colormap={blackwhite}{gray(0cm)=(.25); gray(1cm)=(1)}
}
\usepackage{subfig}
\usetikzlibrary{hobby}

\newcommand{\R}{\mathbb{R}}

\newcommand{\abs}[1]{\left\vert#1\right\vert}
\def\({\left(}
\def\){\right)}

\newcommand{\one}{\mathbbm{1}}

\newcommand{\ep}{\varepsilon}

\newtheorem{thm}{Theorem}[section]
\newtheorem{prop}[thm]{Proposition}

\newtheorem{lem}[thm]{Lemma}

\theoremstyle{definition}

\newtheorem{rem}[thm]{Remark}

\numberwithin{equation}{section}

\allowdisplaybreaks

\author[S.~Patrizi]{\href{http://stepatrizi.altervista.org/}{Stefania Patrizi} }

\address{Department of Mathematics\\
The University of Texas at Austin\\
2515 Speedway, Austin\\
TX 78712, United States of America}
\email{spatrizi@math.utexas.edu}

\author[M. Vaughan]{\href{https://maryvaughan.github.io/}{Mary Vaughan}}

\address{Department of Mathematics and Statistics\\
The University of Western Australia\\
35 Stirling HWY\\
Crawley WA 6009, Australia}
\email{mary.vaughan@uwa.edu.au}



\begin{document}

\title[Convergence result for nonlocal phase field models]{A convergence result for the derivation of front propagation in nonlocal phase field models}

\begin{abstract}
We prove that the mean curvature of a smooth surface in $\R^n$, $n\geq 2$, arises as the limit of a sequence of functions that are intrinsically related to the difference between an $n$- and $1$-dimensional fractional Laplacian of a phase transition. Depending on the order of the fractional Laplace operator, we recover the fractional mean curvature or the classical mean curvature of the surface. Moreover, we show that this is an essential ingredient for deriving the evolution of fronts in fractional reaction-diffusion equations such as those for atomic dislocations in crystals. 
\end{abstract}


\maketitle

\section{Introduction}

In this paper, we prove that the fractional mean curvature and classical mean curvature of a smooth surface in $\R^n$, $n\geq 2$,  arises
as the $\ep$-limit of a sequence of functions $\bar{a}_\ep(x)$, $\ep>0$ (see \eqref{eq:abar}). 
These functions $\bar{a}_\ep$ appear in the study of nonlocal reaction-diffusion equations from an interaction between an $n$-dimensional and a $1$-dimensional fractional Laplacian of a phase transition $\phi$, and thus play a key role in deriving the evolution of interfaces in nonlocal phase field models, see Section \ref{sec:motivation}. 
The convergence result was first observed by Imbert and Souganidis in their 2009 unpublished preprint \cite{Imbert}. 
Since it is foundational for further study on fractional reaction-diffusion equations, we meticulously prove convergence to the mean curvature using a different approach than in \cite{Imbert}. 

Before presenting the main result, we describe the setting of the problem mathematically. 
Let $\Omega$ be a connected set in $\R^n$, $n \geq 2$, with smooth boundary $\Gamma := \partial \Omega$. 
Let $d = d(x)$ be the signed distance function associated to $\Omega$ given by
\begin{equation}\label{eq:signed-distance}
d(x)
	= \begin{cases}
	d(x, \Gamma) & \hbox{if}~x \in \Omega\\
	-d(x,\Gamma) & \hbox{otherwise}. 
	\end{cases}
\end{equation}
For $\rho>0$, define the neighborhood $Q_\rho$ of $\Gamma$ by
\[
Q_\rho = \{x\in \R^n : |d(x)| < \rho\}.
\]
We assume there is some $\rho>0$ such that
\begin{equation}\label{eq:d-assumption}
 d~\hbox{is smooth in}~Q_{2\rho};  \quad \hbox{in particular,}~|\nabla d|=1~\hbox{in}~Q_{2\rho}. 
\end{equation}

\begin{rem}
In what follows, one can replace $d$ by any Lipschitz function $\tilde{d}$ such that $\tilde{d} \equiv d$ in $Q_{\rho}$ and $\tilde{d}$ is smooth in $Q_{2\rho}$. 
\end{rem}

Let $0 < s < 1$ be fixed throughout the paper. 
We consider the fractional mean curvature of order $2s$ of the surface $\Gamma$ as developed in \cite{Imbert09}. Towards this end, define the singular measure
\[
\nu(dz) = \frac{dz}{|z|^{n+2s}},
\]
and set
\begin{align*}
\kappa[x, d] 
	&= 
\nu(\{z : d(x+z) > d(x),~\nabla d(x) \cdot z < 0\}) \\
	&\quad - \nu(\{z : d(x+z) < d(x),~\nabla d(x) \cdot z > 0\}). 
\end{align*}
By \eqref{eq:d-assumption},    $\kappa[x,d]$ is finite in $Q_\rho$ precisely when $s \in (0, \frac12)$, see \cite[Lemma 1]{Imbert09}, and it is the fractional mean curvature at $x \in \Gamma$. 
When $s \in [\frac12,1)$, we instead consider the classical mean curvature of the surface $\Gamma$, see for instance \cite{GilbargTrudinger}. 
Let us simply recall that, since $\Gamma$ is smooth, the mean curvature at $x \in \Gamma$ is well defined and is given by $ -\Delta d(x) /(n-1)$. 
For $n \geq 1$, let $\mathcal{I}_n^s = -c_{n,s}(-\Delta)^s$ denote, up to a constant, the fractional Laplacian of order $2s$ in $\R^n$. 
Specifically, the operator $\mathcal{I}_n^s$ is a nonlocal integro-differential operator given by  
\begin{equation}\label{eq:operator}
\mathcal{I}_n^s u(x) 
	= \operatorname{P.V.}~ \int_{\R^n} \( u(x+y) - u(x)\) \,\frac{dy}{\abs{y}^{n+2s}}, \quad x \in \R^n,
\end{equation}
where $\operatorname{P.V.}$ indicates that the integral is taken in the principal value sense.
Assume that $W$ is a double-well potential satisfying
\begin{equation}\label{eq:W}
\begin{cases}
W \in C^{2, \beta} (\R) & \hbox{for some}~0 < \beta <1 \\
W(0) = W(1) = 0\\
W(u)>0 & \hbox{for any}~u \in (0,1)\\
W'(0)=W'(1)=0\\
W''(0)=W''(1) >0.
\end{cases}
\end{equation}
The phase transition $\phi :\R \to (0,1)$ is then the unique solution to 
\begin{equation} \label{eq:standing wave}
\begin{cases}
 \mathcal{I}_1^s[\phi] = W'(\phi) & \hbox{in}~\R\\
 \dot{\phi}>0 & \hbox{in}~\R\\
\phi(-\infty) = 0, \quad \phi(+\infty)=1,\quad \phi(0) = \frac{1}{2},
\end{cases}
\end{equation}
where $\mathcal{I}_1^s$ denotes the nonlocal operator in \eqref{eq:operator} with $n=1$. See Section \ref{sec:phi} for more on $\phi$. 

Define next the function $a_{\ep}=a_{\ep}(\xi; x)$ by
\begin{equation}\label{aepsilondef}
a_{\ep}(\xi; x)= \int_{\R^n} \( \phi\(\xi + \frac{d(x+\ep z)- d(x)}{\ep}\) - \phi\( \xi + \nabla d(x) \cdot z\) \) \frac{dz}{\abs{z}^{n+2s}},
\end{equation}
where $(\xi,x) \in \R \times \R^n$ and $d$ is in \eqref{eq:signed-distance}. 
Roughly speaking, one may view $a_{\ep}$ as a nonlocal operator acting on $\phi = \phi(\xi)$.

\begin{lem}\label{lem:ae-frac}
For all $x\in Q_\rho$, it holds that
\[
a_{\ep}\left(\frac{d(x)}{\ep}; x\right)
	=  \ep^{2s} \mathcal{I}_n^s \left[ \phi \left( \frac{d(\cdot)}{\ep}\right) \right] (x) - C_{n,s} \mathcal{I}_1^s  \phi \left( \frac{d(x)}{\ep}\right) 
\]
where $C_{n,s}>0$ is given in \eqref{eq:Cns}. 
\end{lem}

The lemma is proven in Section \ref{sec:ae-frac}. 
The corresponding function $\bar{a}_{\ep} = \bar{a}_{\ep}(x)$ is given by 
\begin{equation}\label{eq:abar}
\bar{a}_{\ep}(x) =  \frac{1}{\eta_\ep} \int_{\R}a_{\ep}\(\xi;x\) \dot{\phi}(\xi) \, d \xi
\end{equation}
where $x \in \R^n$ and
\begin{equation}\label{eq:scaling}
\eta_\ep 
= \begin{cases}
\ep^{2s} & \hbox{if}~s\in\left(0,\frac12\right)\\[.25em]
\ep |\ln\ep|&  \hbox{if}~s=\frac12\\[.25em]
\ep & \hbox{if}~s\in\left(\frac12,1\right).
\end{cases}
\end{equation}
Our main result is the following. 

\begin{thm}\label{lem:4} 
Let $\Omega \subset \R^n$, $n \geq 2$, be a connected set with smooth boundary $\Gamma = \partial \Omega$. 
Let $d = d(x)$ be as in \eqref{eq:signed-distance} and assume there is some $\rho>0$ such that \eqref{eq:d-assumption} holds.
Then,
\begin{align*}
\lim_{\ep \to 0} \bar{a}_{\ep}(x)
	&=   
	\begin{cases}
	\displaystyle \kappa[x, d] & \hbox{if}~s \in (0,\frac12)\\[.75em]
	\displaystyle \frac{1}{2} \frac{|S^{n-2}|}{n-1} \Delta d(x) & \hbox{if}~s  = \frac12\\[.75em]
	\displaystyle \frac{c_\star}{2} \frac{|S^{n-2}|}{n-1}\Delta d(x) & \hbox{if}~s  \in (\frac12,1)
	\end{cases}
\end{align*}
uniformly in $Q_{\rho}$ 
where, for $s \in (\frac12,1)$, the constant $c_\star>0$ is given explicitly by
\begin{equation}\label{eq:c-star}
c_\star = \frac{n+2s}{2}\left[\int_{0} ^\infty  \frac{ q^{n-2}((n-1)-q^2)}{\(1+q^2\)^{\frac{n+2s+2}{2}}}\,dq\right] \left[ \int_{\R} \int_{\R} \frac{(\phi(\xi +w) - \phi(\xi))^2}{|w|^{1+2s}} \, d w \, d\xi \right].
\end{equation}
\end{thm}

\begin{rem}
For $n=1$, $\bar{a}_\ep(x) \equiv 0$. Indeed, 
the one-dimensional setting corresponds to $\Omega = (x_0,\infty)$ for some $x_0 \in \R$. Note that $\Gamma = \{x_0\}$, and the signed distance function is precisely $d(x) = x-x_0$. Therefore, $a_\ep(\xi;x) \equiv 0$, and consequently $\bar{a}_\ep(x) \equiv 0$. 
\end{rem}

As addressed above, Theorem \ref{lem:4} was first observed by Imbert and Souganidis, see \cite[Lemma 10]{Imbert} for $s \in (0,\frac12)$ and \cite[Lemma 4]{Imbert} for $s \in [\frac12,1)$. 
While the general ideas of the proof are explained in \cite{Imbert}, we found that certain aspects of their proof require further clarification and additional rigor. 
The critical case $s = \frac12$ is particularly delicate and needs more attention. 
We give our own proof to complement 
 their work and lay the groundwork for continued study on important physical models, such as those outlined below in Section \ref{sec:motivation} for dislocations in crystalline structures. 
 
 The proof of Theorem \ref{lem:4}, especially for the critical case $s = \frac12$, is somewhat technical, but we feel it is more direct than in \cite{Imbert}. Indeed, there are several advantages to our approach. 
First, we carefully decompose the domain of integration in \eqref{eq:abar} to expose the pieces that contribute to the mean curvature in the limit. 
Then, unlike in \cite{Imbert}, we use the known asymptotic behavior of the phase transition $\phi$ and its derivatives at $\pm \infty$ (see Lemma \ref{lem:asymptotics}) to carefully obtain precise error estimates. 
Kindly note that the asymptotic behavior (as well as existence and uniqueness) of the phase transition in \cite[Equation (13)]{Imbert} is assumed and differs from our setting, see \eqref{eq:asymptotics for phi}. 

We remark that the local counterpart of Theorem \ref{lem:4} corresponding to $s=1$ is trivial. In light of Lemma \ref{lem:ae-frac}, one formally computes
\[
\ep^2\Delta\left[ \phi \(\frac{d(\cdot)}{\ep}\)\right](x) -\ddot{\phi} \left(\frac{d(x)}{\ep}\right)
	= \ep \dot{\phi} \left( \frac{d(x)}{\ep}\right) \Delta d(x), \quad \hbox{for}~x \in Q_\rho,
\]
and then defines $a_\ep(\xi;x) :=  \ep \dot{\phi} \left( \xi \right) \Delta d(x)$. Defining $\bar{a}_\ep(x)$ accordingly with $\eta_\ep = \ep$, we trivially recover $\bar{a}_\ep(x) \equiv \Delta d(x)$.
We refer the reader to \cite{Barles-DaLio,  BSS, Barles-Souganidis, Chen} for related problems.

\subsection{Application to nonlocal phase field models}\label{sec:motivation}

We now show how Theorem \ref{lem:4} is crucial for mathematical analysis of nonlocal phase field problems. 
Let us briefly describe the evolutionary problem and present some formal computations.

Consider a connected set $\Omega_0 \subset \R^n$, $n \geq 2$, with smooth boundary $\Gamma_0 = \partial \Omega_0$. 
Let $d_0 = d_0(x)$ denote the signed distance function to $\Gamma_0$ (recall \eqref{eq:signed-distance}) and assume there is some $\rho>0$ such that \eqref{eq:d-assumption} holds for $d_0$. 
Let $u^\ep= u^\ep(t,x)$ denote the solution to the fractional Allen--Cahn equation
\begin{equation}\label{eq:PDE}
 \ep \partial_t u^\ep = \frac{1}{\eta_\ep}(\ep^{2s} \mathcal{I}_n^s[u^\ep] - \widetilde{W}'(u^\ep)) \quad \hbox{in}~(0,\infty) \times \R^n
 \end{equation}
 with $\eta_\ep$ defined in \eqref{eq:scaling} and initial condition
 \begin{equation}\label{eq:PDE-initial}
 u^\ep(0,x) = \phi\left( \frac{d_0(x)}{\ep}\right) \quad\hbox{on}~\R^n,
\end{equation}
where $\phi$ solves \eqref{eq:standing wave} and $\widetilde{W}:= C_{n,s}W$ for an explicit constant $C_{n,s}>0$ (given in \eqref{eq:Cns}). 
Notice that if $W$ satisfies \eqref{eq:W}, then so does $\widetilde{W}$. 

When $n=2$ and $s = \frac12$, the PDE \eqref{eq:PDE} is a rescaled version of the evolutionary Peierls--Nabarro model for atomic dislocations in crystalline structures, see \cite[Section 1.2]{PatriziVaughan} and for the original model \cite{PN1,PN2, PN3}. 
For  the one-dimensional version of \eqref{eq:PDE}, we also refer to \cite{DipierroFigalliVald, DipierroPalatucciValdinoci, DipierroPatriziValdinoci, GonzalezMonneau,patval1}
and the references therein. 
The set $\Gamma_0$ is understood as the initial dislocation curve in the crystal, and 
the parameter $\ep>0$ represents the scaling between the microscopic scale and the mesoscopic scale. 
We send $\ep \to 0$ in \eqref{eq:PDE} to describe the evolution $(\Gamma_t)_{t \geq 0}$ of the dislocation curve $\Gamma_0$ at  the larger length scale. 
Indeed,  as $\ep \to 0$, the function $u^\ep$ approaches a piecewise function with plateaus corresponding to the global minima of $W$ (i.e.~0 and 1) and  whose jump set at time $t$ is $\Gamma_t$. 
Here, we formally derive the evolution of $(\Gamma_t)_{t \geq 0}$ using Theorem \ref{lem:4}. 

Heuristically, let $(\Omega_t)_{t  \geq 0}$ denote the evolution of the set $\Omega_0$ according to \eqref{eq:PDE}, and assume that the boundaries $\Gamma_t = \partial \Omega_t$ are smooth. 
Let $d = d(t,x)$ denote the signed distance function to $\Omega_t$ and assume that \eqref{eq:d-assumption} holds in 
\[
Q_{2\rho} = \{(t,x) \in [0,T] \times \R^n : |d(t,x)| <2 \rho\}, 
\]
for some $T>0$. The formal ansatz for deriving the evolution of the fronts $(\Gamma_t)_{t \in [0,T]}$ is 
\[
u^\ep(t,x)
	\simeq \phi\left( \frac{d(t,x)}{\ep}\right).
\] 
Plugging the ansatz into the equation, 
the left-hand side of \eqref{eq:PDE} gives
\[
\ep \partial_t u^\ep(t,x) = \dot{\phi}\left( \frac{d(t,x)}{\ep}\right) \partial_td(t,x). 
\]
Regarding the right-hand side, we find, for $(t,x)\in Q_\rho$,
\begin{align*}
\ep^{2s} \mathcal{I}_n^s[u^\ep(t,\cdot)](x) - \widetilde{W}'(u^\ep(t,x))
&\simeq \ep^{2s} \mathcal{I}_n^s\left[ \phi\left( \frac{d(t,\cdot)}{\ep}\right)\right](x)- C_{n,s}W'\left(\phi \left( \frac{d(t,x)}{\ep}\right)\right)\\
&= \ep^{2s} \mathcal{I}_n^s\left[ \phi\left( \frac{d(t,\cdot)}{\ep}\right)\right](x)- C_{n,s}\mathcal{I}_1^s [\phi]\left( \frac{d(t,x)}{\ep}\right)\\
&= a_\ep\left( \frac{d(t,x)}{\ep}; t,x\right)
\end{align*}
where the last line follows from Lemma \ref{lem:ae-frac}. 
Note that $a_\ep = a_\ep(\xi;t,x)$ in \eqref{aepsilondef} here depends on $t$ though $d(t,x)$.

Freeze a point $(t,x) \in Q_\rho$. 
Let $\xi = d(t,x)/\ep$ and assume separation of scales. 
That is, assume that $\xi$ and $(t,x)$ are unrelated. 
Then, multiplying the PDE in \eqref{eq:PDE} by $\dot{\phi}(\xi)$ and integrating over $\xi \in \R$ gives
\[
\int_{\R}  (\ep \partial_tu^\ep) \dot{\phi}(\xi) \, d \xi  = \frac{1}{\eta_\ep} \int_{\R} (\ep^{2s} \mathcal{I}_n^s[u^\ep] - \widetilde{W}'(u^\ep)) \dot{\phi}(\xi) \, d \xi.
\]
From the above computations and \eqref{eq:abar}, this yields
\[
\partial_td(t,x) \int_{\R}  [\dot{\phi}(\xi)]^2 \, d \xi 
	\simeq\frac{1}{\eta_\ep} \int_{\R} a_\ep\left(\xi; t,x\right)\, d \xi = \bar{a}_\ep(t,x).
\]
Therefore, by Theorem \ref{lem:4}, we conclude that, in $Q_\rho$, 
\[
\partial_td(t,x) \simeq  \left[\int_{\R}  [\dot{\phi}(\xi)]^2 \, d \xi\right]^{-1} \bar{a}_\ep(t,x)
	\simeq \mu  \begin{cases}
	\kappa[x,d(t,\cdot)] & \hbox{if}~s \in (0,\frac12)\\[.5em]
	\Delta d(t,x) & \hbox{if}~s \in [\frac12,1)
	\end{cases}
\]
for a constant $\mu>0$ depending on $n$ and $s$. 
This formally shows that the forming interphases as $\ep\to0$,  $(\Gamma_t)_{t \in [0,T]}$, move according to  either their fractional or classical mean curvature. 

Of course, these are formal computations that need to be rigorously checked. 
We reference the reader to \cite{PatriziVaughan} for complete heuristics and a study of \eqref{eq:PDE} in the case $s = \frac12$. 
There, our initial configuration is a superposition of functions of the form \eqref{eq:PDE-initial}  which corresponds to a finite collection of dislocations, and we show using Theorem \ref{lem:4} that they move independently and according to their mean curvature. 
The cases $s \in (0,\frac12)$ and $s \in (\frac12,1)$ will be treated separately in future work. 
Problem \eqref{eq:PDE} was previously studied in   \cite{Imbert} for any $s\in(0,1)$, for the case of one dislocation curve (i.e.~for the initial condition \eqref{eq:PDE-initial}), 
and under the additional assumption of the existence of certain correctors, see Assumption 3 there.

The stationary problem was studied in \cite{AmbrosioDePhilippisMartinazzi} and \cite{SavinValdinoci} where they prove that the fractional Allen--Cahn energy, with the same scaling as in \eqref{eq:scaling}, $\Gamma$-converges to the $2s$-fractional perimeter functional when $s \in (0,\frac12)$ and classical perimeter functional when $s \in [\frac12,1)$.
The local problem in which \eqref{eq:PDE} is instead driven by the usual Laplacian $\Delta$ was studied by Modica--Mortola \cite{MM} for the stationary problem and Chen \cite{Chen} for the evolutionary problem.

\subsection{Organization of the paper}

The rest of the paper is organized as follows. First, in Section \ref{sec:phi}, we review some properties of the phase transition $\phi$. 
Preliminaries on $a_\ep$ and the proof of Lemma \ref{lem:ae-frac} are presented in Section \ref{sec:ae-frac}. 
Section \ref{sec:prelim} contains preliminaries for the Proof of Theorem \ref{lem:4} when $s \in [\frac12,1)$.
We prove Theorem \ref{lem:4} for $s = \frac12$, $s \in (\frac12,1)$, and $s \in (0,\frac12)$, respectively, in Sections \ref{sec:s-critical}, \ref{sec:s-large}, and \ref{sec:s-small}. 

\subsection{Notations}

In the paper, we will denote by $C>0$ any universal constant depending only on the dimension $n$, $s$, and $W$. 

Denote by $S^n$ the unit sphere in $\R^{n+1}$ and $\mathcal{H}^n$ the $n$-dimensional Hausdorff measure. 

Given a function $u = u(x)$, defined on a set $A$,  we write $u = O(\ep)$ if there is $C>0$ such that
$|u(x)| \leq C \ep$ for all $x\in A$, and we write $u = o_\ep(1)$ if  $\lim_{\ep \to 0} u(x) = 0$, uniformly in $x\in A$. 

For a set $A$, we denote by $\one_A$ the characteristic function of the set $A$. 

\section{The phase transition $\phi$}\label{sec:phi}

In this section, we present background and preliminary results on the phase transition $\phi$.
Let $H(\xi)$ denote the Heavyside function. 

\begin{lem}\label{lem:asymptotics}  
There is a unique solution $\phi \in C^{2, \beta}(\R)$ of \eqref{eq:standing wave}, with $\beta$ as in \eqref{eq:W}. Moreover, there exists a constant $C>0$ 
and $\kappa>2s$ (only depending on~$s$) such that
\begin{equation}\label{eq:asymptotics for phi}
\abs{\phi(\xi) - H(\xi) +  \frac{1}{2sW''(0)} \frac{\xi}{|\xi|^{1+2s}}} \leq \frac{C}{\abs{\xi}^{\kappa}}, \quad \text{for }\abs{\xi} \geq 1,
\end{equation}
and
\begin{equation}\label{eq:asymptotics for phi dot}
\frac{1}{C\abs{\xi}^{2s+1}}\leq \dot{\phi}(\xi) \leq \frac{C}{\abs{\xi}^{2s+1}},\quad |\ddot{\phi}(\xi)| \leq \frac{C}{\abs{\xi}^{2s+1}}
 \quad \text{for }\abs{\xi} \geq 1.
\end{equation}
If $s\in(\frac12,1)$, then 
\begin{equation}\label{finite-energy}c_1:=\frac12\int_\R\int_\R\frac{(\phi(\xi+t)-\phi(\xi))^2}{|t|^{1+2s}}\,dt\,d\xi<\infty.
\end{equation}
\end{lem}

\begin{proof}
The existence of a unique solution of \eqref{eq:standing wave} is proven in  \cite{CabreSola} for $s=\frac12$ and in 
\cite{CabreSire,PalatucciSavinValdinoci} for any $s\in(0,1)$ together with \eqref{finite-energy} when $s>\frac12$. 
Estimate \eqref{eq:asymptotics for phi} is  proven 
in \cite{GonzalezMonneau} for  $s=\frac{1}{2}$  and 
in \cite{DipierroFigalliVald}  and  \cite{DipierroPalatucciValdinoci}, respectively, when  $s\in\left(0,\frac{1}{2}\right)$ and $s\in\left(\frac{1}{2},1\right)$. 
Finally, estimates \eqref{eq:asymptotics for phi dot} are shown in \cite{CabreSire,MonneauPatrizi,PalatucciSavinValdinoci,patval2}. 
\end{proof}

We will use several times that, by \eqref{eq:asymptotics for phi}, if $M>1$, 
\begin{equation}\label{phi-infyinftyaeptermxiestbis}
\int_{-M}^M
\dot\phi(\xi)\,d\xi=\phi(M)-\phi(-M)=1+O(M^{-2s}),
\end{equation}
and 
\begin{equation}\label{phi-infyinftyaeptermxiest}
\int_{\{|\xi|>M\}}\dot\phi(\xi)\,d\xi=1-  \int_{-M}^M\dot\phi(\xi)\,d\xi=O(M^{-2s}).
\end{equation}

\section{The function $a_\ep$ and fractional Laplacians of $\phi$}\label{sec:ae-frac}

Here, we present some preliminaries on $a_\ep$ and prove Lemma \ref{lem:ae-frac}.

We will use the following lemma throughout the paper without reference. 
The proof is a standard computation in polar coordinates.  

\begin{lem}
There exist $C_1,\, C_2>0$ such that for any $R>0$,
$$\int_{\{|z|<R\}}\frac{dz}{|z|^{n+2s-2}}= C_1 R^{2-2s}
\quad \hbox{and} \quad \int_{\{|z|>R\}}\frac{dz}{|z|^{n+2s}}= \frac{C_2}{R^{2s}}.$$
\end{lem}

Accordingly, by the regularity of $\phi$ and $d$, the integral in \eqref{aepsilondef} is well defined for $x \in Q_\rho$:
\begin{align*}
\int_{\R} &\left| \phi\(\xi + \frac{d(x+\ep z)- d(x)}{\ep}\) - \phi\( \xi + \nabla d(x) \cdot z\) \right| \frac{dz}{\abs{z}^{n+2s}}\\
&\leq C \left[\int_{\{|z|<1\}} \frac{|d(x+\ep z) - d(x) -\nabla d(x) \cdot (\ep z)|}{\ep} \frac{dz}{\abs{z}^{n+2s}}
	+  \int_{\{|z|>1\}}\frac{dz}{\abs{z}^{n+2s}}\right]\\
&\leq C \left[\int_{\{|z|<1\}} \ep |z|^2 \frac{dz}{\abs{z}^{n+2s}}
	+  \int_{\{|z|>1\}}\frac{dz}{\abs{z}^{n+2s}}\right] \leq C. 
\end{align*}

We will need the next result that allows us to view one-dimensional fractional Laplacians of functions defined over $\R$ equivalently as $n$-dimensional fractional Laplacians. 

\begin{lem}\label{lem:one to n}
For a vector $e \in \R^n$ and a function $v\in C^{1,1}(\R)$, let $v_e(x) = v(e\cdot x): \R^n \to \R$. Then,
\[
\mathcal{I}_n^s[v_e](x) =|e|^{2s} C_{n,s} \mathcal{I}_1^s[v](e\cdot x)
\]
where 
\begin{equation}\label{eq:Cns}
C_{n,s} =  \int_{\R^{n-1}} \frac{1}{(\abs{y}^2 + 1)^{\frac{n+2s}{2}} } \, dy.
\end{equation}
Consequently,
\begin{equation*}\label{eq:one to n}
|e|^{2s} C_{n,s} \mathcal{I}_1^s[v](\xi) =\operatorname{P.V.} \int_{\R^n} \( v(\xi + e \cdot z) - v(\xi)\) \frac{dz}{\abs{z}^{n+2s}}, \quad \xi \in \R.
\end{equation*}
\end{lem}

\begin{proof}
The case $e=0$ is trivial. Therefore, let us assume $e\neq0$ and let $c:=|e|>0$. 
Begin by writing
\[
\mathcal{I}_n^s[v_e](x) 
	=\operatorname{P.V.}\int_{\R^n} \(v(e \cdot x + e \cdot z) - v(e \cdot x)\) \frac{dz}{\abs{z}^{n+2s}}.
\]
By rotation, it is enough to prove the result for $e = ce_1$. 
Observe for $x = (x_1,x') \in \R \times \R^{n-1}$ that
\begin{align*}
\mathcal{I}_n^s[v_{ce_1}](x)
	&=\operatorname{P.V.}\int_{\R^n} \(v(cx_1 +cz_1) - v(cx_1)\) \frac{dz}{\abs{z}^{n+2s}}\\
	&=\operatorname{P.V.} \int_{\R} \(v(cx_1 +cz_1) - v(cx_1)\) \( \int_{\R^{n-1}} \frac{1}{\abs{(z_1,z')}^{n+2s}} \, dz' \) \, dz_1.
\end{align*}
Since
\begin{align*}
\int_{\R^{n-1}} \frac{1}{\abs{(z_1,z')}^{n+2s}} \, dz'
	&= \int_{\R^{n-1}} \frac{1}{(\abs{z'}^2 + z_1^2)^{\frac{n+2s}{2}} } \, dz'\\
	&= \frac{1}{\abs{z_1}^{n+2s}} \int_{\R^{n-1}} \frac{1}{(\abs{y}^2 + 1)^{\frac{n+2s}{2}} } \, \abs{z_1}^{n-1}dy
	= \frac{C_{n,s}}{\abs{z_1}^{1+2s}},
\end{align*}
we have
\begin{align*}
\mathcal{I}_n^s[v_{ce_1}](x) 
	&= C_{n,s}\operatorname{P.V.}\int_{\R} \(v(cx_1 +cz_1) - v(cx_1)\) \, \frac{dz_1}{\abs{z_1}^{1+2s}}\\&
	= c^{2s}C_{n,s}\operatorname{P.V.}\int_{\R} \(v(cx_1 +z_1) - v(cx_1)\) \, \frac{dz_1}{\abs{z_1}^{1+2s}}
	= c^{2s}C_{n,s} \mathcal{I}_1^s[v](ce_1\cdot x).
\end{align*}
\end{proof}

We are now ready to prove Lemma \ref{lem:ae-frac}.

\begin{proof}[Proof of Lemma \ref{lem:ae-frac}]
First, we write $a_{\ep} = a_{\ep}\(d(x)/\ep ; x\)$ as
\begin{equation}\label{eq:split-for-ae-laplac}
\begin{aligned}
a_{\ep}
	&=\operatorname{P.V.} \int_{\R^n} \( \phi\(\frac{d(x+\ep z)}{\ep}\) - \phi\(\frac{d(x)}{\ep}\) \)\, \frac{dz}{\abs{z}^{n+2s}} \\
	&\quad -\operatorname{P.V.}\int_{\R^n} \(\phi\( \frac{d(x)}{\ep} + \nabla d(x) \cdot z\) - \phi\(\frac{d(x)}{\ep}\)\) \, \frac{dz}{\abs{z}^{n+2s}}.
\end{aligned}
\end{equation}
Since $e = \nabla d(x)$  is well defined when $\abs{d(x)} \leq \rho$ and a unit vector, by applying Lemma \ref{lem:one to n} to the second integral in the right-hand side of  \eqref{eq:split-for-ae-laplac} and a change of variables in the first integral, we obtain
\begin{align*}
a_{\ep}
	&= \ep^{2s} \operatorname{P.V.} \int_{\R^n} \( \phi\(\frac{d(x+y)}{\ep}\) - \phi\(\frac{d(x)}{\ep}\) \)\, \frac{dy}{\abs{y}^{n+1}} 
	 -C_{n,s} \mathcal{I}_1^s[\phi]\(\frac{d(x)}{\ep}\)\\
	&= \ep^{2s} \mathcal{I}_n^s\left[\phi\(\frac{d(\cdot)}{\ep}\) \right](x) - C_{n,s} \mathcal{I}_1^s[\phi]\(\frac{d(x)}{\ep}\).
\end{align*}
\end{proof}

\section{Preliminaries for $s \in [\frac12,1)$}\label{sec:prelim}

We give preliminaries for the proof Theorem \ref{lem:4} when $s \in [\frac12,1)$. 

\subsection{Mean curvature}

Assume that $d$ in \eqref{eq:signed-distance} satisfies \eqref{eq:d-assumption}. 
Then, in $Q_\rho$, the eigenvalues of $D^2d(x)$ are
$$\lambda_i(x)=\frac{-\kappa_i}{1-\kappa_i d(x)}, \, i=1,\ldots,n-1, \quad\lambda_n(x)=0,$$
see, for example, \cite[Lemma 14.17]{GilbargTrudinger}.  
Moreover, since $|\nabla d| = 1$ in $Q_\rho$, we have the equation $D^2d(x)\nabla d(x)=0$ from which we see that $\nabla d(x)$ is an eigenvector, with norm 1, for $D^2d(x)$ with associated eigenvalue $\lambda_n(x) = 0$. 

Let us denote 
\[
 A(y'):=\frac12\sum_{i=1}^{n-1}\lambda_iy_i^2, \quad y' = (y_1,\dots, y_{n-1}) \in \R^{n-1}.
\]
Note, for $x\in Q_\rho$, that
\begin{equation}\label{eq:Atheta-MC}
 \int_{S^{n-2}} A(\theta) \, d \theta 
 	= \frac{1}{2} \sum_{i=1}^{n-1}\lambda_i\int_{S^{n-2}} \theta_i^2 \, d \theta
	=\frac{1}{2} \sum_{i=1}^{n-1}\lambda_i\frac{|S^{n-2}|}{n-1}
	=\frac{1}{2} \frac{|S^{n-2}|}{n-1}\Delta d(x).
\end{equation}

\subsection{Change of variables in $\bar{a}_\ep$}

In light of \eqref{eq:Atheta-MC}, 
we write the following lemma for $\bar{a}_\ep$ as an integral in terms of polar coordinates which exposes $A(\theta)$, $\theta \in S^{n-2}$. 
This will set the stage for proving Theorem \ref{lem:4} when $s \in [\frac12,1)$. 

\begin{lem}\label{lem:ae-change-of-variables}
Let $s \in [\frac12,1)$ and $r \in (0,\frac{\rho}{2})$ be fixed. Then, for all $(t,x) \in Q_\rho$ and $0 < \ep < r$, 
\[
\bar{a}_\ep(t,x)
= O\left(\frac{\ep^{2s}}{r^{2s}\eta_\ep}\right) 
+ \frac{\ep^{2s}}{\eta_\ep} \int_{\R} G(\xi) \dot{\phi}(\xi) \, d \xi
\]
where
\begin{align*}
G(\xi)
&=\int_0^r \frac{dp}{p^{2s+1}} \int_{S^{n-2}} d \theta
	 \int_{-\frac{r}{p}}^{\frac{r}{p}}
 \left( \phi\(\xi + \frac{p}{\ep}\(t+pb(\theta,t,r)\)\) - \phi\( \xi + \frac{tp}{\ep}\)\right)\frac{dt}{ (t^2 +1)^{\frac{n+2s}{2}}} 
\end{align*}
with
\begin{equation}\label{bthetatrde}
b(\theta,t,r):=A(\theta)+O\left( r(1+t^2)\right).
\end{equation}
\end{lem}

\begin{rem}
Notice if $s \in (0,\frac12)$, then $O(\frac{\ep^{2s}}{r^{2s}\eta_\ep}) = O(\frac{1}{r^{2s}})$ is ineffectual for $r$ small. 
\end{rem}

\begin{proof}
Let $T=(v_1,\ldots,v_n)$ be an orthogonal matrix whose columns are a set of orthonormal  eigenvectors $v_1, \ldots,v_n$ for the eigenvalues $\lambda_1,\ldots, \lambda_n$ with 
$v_n=\nabla d(x)$. 
We make the change of variables $\ep z=Ty$ to write
\begin{align*}
a_\ep(\xi;x)
=\ep^{2s} \int_{\R^n}\( \phi\(\xi + \frac{1}{\ep}(d(x+Ty) - d(x))\) - \phi\( \xi + \frac{y_n}{\ep}\) \) \frac{dy}{\abs{y}^{n+2s}}
\end{align*} 
where $y=(y',y_n)$ with $y'\in\R^{n-1}$. 
Then, we split $a_\ep$ as follows, for $0 <\ep < r$,  
\[
a_\ep(\xi;x)
=\int_{\{|y'|, |y_n|<r\}}(\ldots)
+\int_{\{|y'|>r\}\cup\{|y_n|>r\}}(\ldots).
\]
Since $0<\phi<1$, we have that
\begin{align*}
&\ep^{2s} \abs{\int_{\{|y'|>r\}\cup\{|y_n|>r\}}\( \phi\(\xi + \frac{1}{\ep}\(d(x+Ty) - d(x)\)\) - \phi\( \xi + \frac{y_n}{\ep}\) \) \frac{dy}{\abs{y}^{n+2s}}}\\&
\leq 2\ep^{2s}\int_{\{|y|>r\}}\frac{dy}{\abs{y}^{n+2s}}
\leq \frac{C\ep^{2s}}{r^{2s}}.
\end{align*}
Therefore,
\begin{equation}\label{eq:ae-r-split}
a_\ep(\xi;x)
	=O \left(\frac{\ep^{2s}}{r^{2s}}\right) + \ep^{2s} G(\xi)
\end{equation}
where
\[
G(\xi) = \int_{\{|y'|, |y_n|<r\}} \( \phi\(\xi + \frac{1}{\ep}(d(x+Ty) - d(x))\) - \phi\( \xi + \frac{y_n}{\ep}\) \) \frac{dy}{\abs{y}^{n+2s}}.
\]
Now,  if $|y_n|, |y'|< r$, then, $x +Ty \in Q_{2\rho}$ and by construction of $T$ and the regularity of $d$, we have
\begin{align*}
d(x+Ty) - d(x) 
	&= \nabla d(x) \cdot Ty+\frac{1}{2}D^2d(x) Ty\cdot Ty+O(r|y|^2) \\
	&= y_n + A(y') + O(r|y|^2).
\end{align*}
Consequently, 
\[
G(\xi) = \int_{\{|y'|, |y_n|<r\}} \( \phi\(\xi + \frac{1}{\ep}(y_n + A(y') + O(r|y|^2)\) - \phi\( \xi + \frac{y_n}{\ep}\) \) \frac{dy}{\abs{y}^{n+2s}}.
\]

Next, we make the further change of variable $t = y_n/|y'|$ to write
\begin{align*}
 G(\xi)
&=   \int_{\{|y'|<r\}} \frac{dy'}{|y'|^{n+2s}} 
	 \int_{\{|y_n|<r\}}\frac{dy_n}{ \left( \frac{y_n^2}{|y'|^2} +1\right)^{\frac{n+2s}{2}}}\\
&\qquad \times \( \phi\(\xi + \frac{1}{\ep}\(y_n+A(y')+O(r|y|^2)\)\) - \phi\( \xi + \frac{y_n}{\ep}\) \)\\
&=   \int_{\{|y'|<r\}} \frac{dy'}{|y'|^{n+2s-1}} 
	 \int_{\{|t|<\frac{r}{|y'|}\}}\frac{dt}{ (t^2 +1)^{\frac{n+2s}{2}}}\\
&\qquad\times\( \phi\(\xi + \frac{|y'|}{\ep}\(t+|y'|A\left(\frac{y'}{|y'|}\right)+|y'|O(r(1+t^2))\)\) - \phi\( \xi + \frac{t|y'|}{\ep}\) \).
\end{align*}
Finally, using polar coordinates $y' = p \theta$ with $p>0$ and $\theta \in S^{n-2}$, we write
\begin{align*}
G(\xi)
&=  \int_0^r \frac{dp}{p^{2s+1}} \int_{S^{n-2}} d \theta
	 \int_{\{|t|<\frac{r}{p}\}}\frac{dt}{ (t^2 +1)^{\frac{n+2s}{2}}}\\
&\qquad\times\( \phi\(\xi + \frac{p}{\ep}\(t+pA(\theta)+pO(r(1+t^2))\)\) - \phi\( \xi + \frac{tp}{\ep}\) \).
\end{align*}
With \eqref{eq:ae-r-split} and recalling that $\int_{\R} \dot{\phi}(\xi)\, d\xi=1$, this completes the proof.
\end{proof}

\section{Proof of Theorem \ref{lem:4} for $s = \frac12$}\label{sec:s-critical}

Throughout this section, assume that $s = \frac12$.

Fix $r \in(0,\frac{\rho}{2})$ and let $0 < \ep < r$. By Lemma \ref{lem:ae-change-of-variables} we have
\begin{equation}\label{eq:aebarI}
\bar{a}_\ep(x)
= O\left(\frac{1}{r | \ln \ep|}\right) + \frac{1}{| \ln \ep|} \int_{\R} G(\xi) \dot{\phi}(\xi) \, d \xi
\end{equation}
where
\begin{align*}
G(\xi)
&=\int_0^r \frac{dp}{p^{2}} \int_{S^{n-2}} d \theta
	 \int_{-\frac{r}{p}}^{\frac{r}{p}}
 \left( \phi\(\xi + \frac{p}{\ep}\(t+pb(\theta,t,r)\)\) - \phi\( \xi + \frac{tp}{\ep}\)\right)\frac{dt}{ (t^2 +1)^{\frac{n+1}{2}}}.
\end{align*}
By the regularity of $\phi$ and recalling \eqref{bthetatrde}, we see that
\begin{align*}
 &\left|\int_0^\ep \frac{dp}{p^{2}} \int_{S^{n-2}} d \theta
	 \int_{-1}^1 \(\phi\(\xi + \frac{p}{\ep}\(t+pb(\theta,t,r)\)\) - \phi\( \xi + \frac{tp}{\ep}\)\)\,\frac{dt}{ (t^2 +1)^{\frac{n+1}{2}}} \right|\\
	 &\quad\leq \frac{C}{\ep} \int_0^\ep \, dp \int_{S^{n-2}} d \theta
	 \int_{-1}^1 |A(\theta)+O(r(1+t^2))| \frac{dt}{ (t^2 +1)^{\frac{n+1}{2}}}\\
	& \quad \leq  \frac{C}{\ep}\int_0^\ep \, dp=C
\end{align*}
and also
\begin{align*}
 &\left|\int_0^\ep \frac{dp}{p^{2}} \int_{S^{n-2}} d \theta
	 \int_{\{1< |t|<\frac{r}{p}\}} \( \phi\(\xi + \frac{p}{\ep}\(t+pb(\theta,r,t)\)\) - \phi\( \xi + \frac{tp}{\ep}\)\)\frac{dt}{ (t^2 +1)^{\frac{n+1}{2}}}\right|\\
	&\quad\leq  \frac{C}{\ep}\int_0^\ep \, dp \int_{S^{n-2}} d \theta
	 \int_{\{1<|t|<\frac{r}{p}\}} (1+rt^2) \frac{dt}{ |t|^{n+1}}\\
	 &\quad\leq  
	 \frac{C}{\ep}\int_0^\ep \, dp 
	 \left[ \int_{\{|t|>1\}} \frac{dt}{|t|^{n+1}}  + r \int_{\{1 < |t| < \frac{r}{p}\}}  \frac{dt}{ |t|}\right]\\
	 &\quad\leq  
	 \frac{C}{\ep}\int_0^\ep 
	 (1 + r \ln r - r \ln p)\,dp
	 \leq (1 + r | \ln \ep|).
\end{align*}
Together with \eqref{eq:aebarI}, we have 
\begin{equation}\label{eq:aebarI-bis}
\bar{a}_\ep(x)
= O\left(\frac{1}{r | \ln \ep|}\right) 
+ O(r) 
+ \frac{1}{| \ln \ep|} \int_{\R} I(\xi) \dot{\phi}(\xi) \, d \xi
\end{equation}
where
\begin{align*}
I(\xi) &= \int_\ep^r \frac{dp}{p^{2}} \int_{S^{n-2}} d \theta
	 \int_{-\frac{r}{p}}^{\frac{r}{p}}  
 \left( \phi\(\xi + \frac{p}{\ep}\(t+pb(\theta,t,r)\)\) - \phi\( \xi + \frac{tp}{\ep}\)\right)\,\frac{dt}{ (t^2 +1)^{\frac{n+1}{2}}}.
\end{align*}

\medskip

\noindent
\underline{\bf Heuristics.}
We will prove that one of the main contributions comes from values of $p$ between $\ep^\frac12$ and $r$ and values of $t$ such that, for $A(\theta)>0$, 
 \begin{equation}\label{tvalues1}
 -p A\left(\theta\right)<t<0,
 \end{equation}
 and for $A(\theta)<0$,
 \begin{equation}\label{tvalues2}
0<t< -p A\left(\theta\right).
 \end{equation}
 Indeed, by \eqref{eq:asymptotics for phi} if $A(\theta)>0$, for points $t$ as in \eqref{tvalues1}, the integrand function in $I(\xi)$ is close to 1 and thus
 \begin{align*}\frac{1}{|\ln\ep|}\int_\R d\xi\, \dot\phi(\xi)\int_{\ep^\frac12}^r\frac{dp}{p^2}\int_{-p A(\theta)}^0(\ldots)\,dt&\simeq \frac{1}{|\ln\ep|}\int_\R d\xi \,\dot\phi(\xi)\int_{\ep^\frac12}^r\frac{dp}{p^2}\int_{-p A(\theta)}^0dt\\&
 =\frac{A(\theta)}{|\ln\ep|}\int_\R d\xi\, \dot\phi(\xi)\int_{\ep^\frac12}^r\frac{dp}{p}\\&\simeq\frac{A(\theta)}{2},\end{align*}
 see \eqref{eq:J13} below. Similarly, if $A(\theta)<0$, for points  $t$ as in \eqref{tvalues2} the integrand function
   is close to $-1$ and 
   $$\frac{1}{|\ln(\ep)|}\int_\R d\xi\, \dot\phi(\xi)\int_{\ep^\frac12}^r\frac{dp}{p^2}\int_0^{-p A(\theta)}(\ldots)\,dt\simeq \frac{A(\theta)}{2}.$$
The other main contribution comes from values of $p$ between $\ep$ and $\ep^\frac12$ and values of $t$ between -1 and 1.
Indeed, we will show that  
\begin{align*}\frac{1}{|\ln\ep|}\int_\R d\xi\, \dot\phi(\xi)\int_{\ep}^{\ep^\frac{1}{2}}\frac{dp}{p^2}\int_{-1}^{1}(\ldots)
\,dt&\simeq \frac{A(\theta)}{|\ln\ep|}\int_\R d\xi\, \dot\phi(\xi)\int_\ep^{\ep^\frac12}\frac{dp}{p}=\frac{ A(\theta)}{2}, \end{align*}
see \eqref{J_1estimatefinal} and \eqref{J_5estimatefinal} below. 

To formally prove the estimates above, we start by splitting $I(\xi)$ as
 \begin{equation}
 \label{Ixisplitintbo}
 \begin{aligned}
 I(\xi)
 &=\int_{\ep}^r\frac{dp}{p^2}\int_{S^{n-2}}d\theta\int_{-1}^1(\ldots)\,dt
 	+\int_{\ep}^r\frac{dp}{p^2}\int_{S^{n-2}}d\theta\int_{\{1<|t|<\frac{r }{p}\}}(\ldots)\,dt\\
&=:I_1(\xi)+I_2(\xi),
\end{aligned}
  \end{equation}
 and then estimate $\frac{1}{|\ln \ep|} \int_{\R} \dot{\phi}(\xi) I_1(\xi) \, d \xi$ and $\frac{1}{|\ln \ep|} \int_{\R} \dot{\phi}(\xi) I_2(\xi) \, d \xi$ separately. 
 
 \medskip
 
 \noindent
\underline{\bf Step 1. Estimating $\frac{1}{|\ln \ep|} \int_{\R} \dot{\phi}(\xi) I_1(\xi) \, d \xi$}.
We will show that 
\begin{equation}\label{eq:I_1-final}
 \frac{1}{|\ln \ep|} \int_{\R} \dot{\phi}(\xi) I_1(\xi) \, d \xi
	=\int_{S^{n-2} }A(\theta)\,d\theta+o_\ep(1) + o_r(1),
\end{equation}
where $o_\ep(1)$ depends on the parameter  $r$. 

Note that if $|t|<1$, then for some $C_0>0$, 
\begin{equation}\label{1=t^2termysmall}
|b(\theta, r,t)-A(\theta)|\leq  C_0r. 
\end{equation}
Let  $r$ and $\delta$ be such that 
\begin{equation}\label{deltaC_0-ine}
C_0r\leq \delta<\frac12.
\end{equation} 
We write
\begin{equation}
\begin{aligned}
I_1(\xi)&=
\int_{S^{n-2}\cap\{ A(\theta)>3\delta\} }d\theta\,(\ldots)+\int_{S^{n-2}\cap\{ A(\theta)<-3\delta\} }d\theta\,(\ldots)+\int_{S^{n-2}\cap\{ |A(\theta)|\leq3\delta\} }d\theta\,(\ldots)\\&
 =:I_1^1(\xi)+I_1^2(\xi)+I_1^3(\xi).
\end{aligned}
 \label{Ixisplitint}
 \end{equation}
 Beginning with $I_1^1(\xi)$, we further split, for $R>1$ to be chosen,
 \begin{equation}\label{J1J2splitofI_1}\begin{split}
I_1^1(\xi)
 	&= \int_{S^{n-2}\cap\{ A(\theta)>3\delta\} }d\theta \int_{(R\ep)^{\frac{1}{2}}}^r dp\int_{-1}^1(\ldots)\,dt\\
	&\qquad+ \int_{S^{n-2}\cap\{ A(\theta)>3\delta\} }d\theta \int_{\ep}^{(R\ep)^{\frac{1}{2}}}dp\int_{-1}^1 (\ldots)\,dt\\
	&=: J_1(\xi) + J_2(\xi).
 \end{split}
 \end{equation}
 
 In what follows, we will use several times without reference that, recalling \eqref{1=t^2termysmall} and \eqref{deltaC_0-ine}, if $A(\theta)>3\delta$ and $|t|<1$, then $b(\theta,t,r)>0$ and  by the monotonicity of $\phi$, 
 $$\phi\(\xi + \frac{p}{\ep}\left(t+p b(\theta, t, r)\right)\) - \phi\( \xi + \frac{tp}{\ep}\)>0.$$
 
 \medskip
 
 \noindent
\underline{\bf Step 1a. Estimating $\frac{1}{|\ln \ep|} \int_{\R} \dot{\phi}(\xi) J_1(\xi) \, d \xi$}.
We will show that
\begin{equation}\label{J1}
\frac{1}{|\ln \ep|} \int_{\R} \dot{\phi}(\xi) J_1(\xi) \, d \xi
	= \frac12\int_{S^{n-2}\cap\{ A(\theta)>3\delta\} }A(\theta)\,d\theta+o_\ep(1)+O(\delta)+O(r).
\end{equation}
Begin by writing
\begin{align*}
J_1(\xi)
&=\int_{S^{n-2}\cap\{ A(\theta)>3\delta\} }d\theta \int_{(R\ep)^{\frac{1}{2}}}^r\frac{dp}{p^2}\int_{\delta p}^1(\ldots)\,dt\\&
\quad +\int_{S^{n-2}\cap\{ A(\theta)>3\delta\} }d\theta \int_{(R\ep)^{\frac{1}{2}}}^r\frac{dp}{p^2}\int_{-\delta p}^{\delta p}(\ldots)\,dt\\&
\quad +\int_{S^{n-2}\cap\{ A(\theta)>3\delta\} }d\theta\int_{(R\ep)^{\frac{1}{2}}}^r\frac{dp}{p^2}\int_{-p (A(\theta)-2\delta)} ^{-\delta p}(\ldots)\,dt\\&
\quad +\int_{S^{n-2}\cap\{ A(\theta)>3\delta\} }d\theta\int_{(R\ep)^{\frac{1}{2}}}^r \frac{dp}{p^2}\int_{-p (A(\theta)+2\delta)}^{-p (A(\theta)-2\delta)}(\ldots)\,dt\\&
 \quad+\int_{S^{n-2}\cap\{ A(\theta)>3\delta\} }d\theta\int_{(R\ep)^{\frac{1}{2}}}^r \frac{dp}{p^2}\int_{-1}^{-p (A(\theta)+2\delta)}(\ldots)\,dt\\
 &=: J_1^1(\xi) + J_1^2(\xi) + J_1^3(\xi) + J_1^4(\xi) + J_1^5(\xi). 
  \end{align*}
Notice that $ p (A(\theta)+2\delta)<1$ for $ p\leq r$ and $r$ small enough. 

The main contribution in $J_1(\xi)$ comes from $J_1^3(\xi)$. Indeed,  we first show that
\begin{equation}\label{eq:J13}
\frac{1}{|\ln \ep|} \int_{\R} \dot{\phi}(\xi) J_1^3(\xi) \, d \xi
	= \frac12\int_{S^{n-2}\cap\{ A(\theta)>3\delta\} }A(\theta)\,d\theta+o_\ep(1)+O(\delta)+O(r).
\end{equation}
Notice that,  for $R\delta>2$ and $\delta$ as in \eqref{deltaC_0-ine},  
if 
\[
\theta \in S^{n-2} \cap \{A(\theta)>3\delta\}, \quad (R\ep)^\frac12 \leq p\leq r, \quad -p(A(\theta)-2\delta) \leq t\leq -p \delta,\quad |\xi|\leq \frac{R\delta}{2},
\]
then
$$\xi + \frac{tp}{\ep}\leq \xi-\frac{\delta p^2}{\ep}\leq \xi-R\delta\leq -\frac{R\delta}{2}$$
and, recalling \eqref{1=t^2termysmall},
$$\xi + \frac{p}{\ep}(t+p b(\theta,t, r))\geq\xi + \frac{p}{\ep}(t+p A(\theta)-p C_0r)\geq\xi+\frac{\delta p^2}{\ep}\geq \xi+R\delta\geq \frac{R\delta}{2}.$$
Consequently, by \eqref{eq:asymptotics for phi},
\begin{align*}
& \phi\(\xi + \frac{p}{\ep}\left(t+p b(\theta,t, r)\right)\) - \phi\( \xi + \frac{tp}{\ep}\)\\&
=H\(\xi + \frac{p}{\ep}\left(t+p b(\theta,t, r)\right)\) -H\( \xi + \frac{tp}{\ep}\) +O\left(\frac{1}{\xi + \frac{p}{\ep}\left(t+p b(\theta, t,r)\right)}\right)+
O\left(\frac{1}{\xi + \frac{tp }{\ep}}\right)\\&
=1+O\(\frac{1}{R\delta}\). 
\end{align*}
 Therefore, 
\begin{equation} \label{eq:split-main-error}
\begin{aligned}
& \int_\R\dot\phi(\xi)J_1^3(\xi) \,d\xi \\
&= \int_{-\frac{R\delta}{2}}^{\frac{R\delta}{2}}d\xi \,\dot\phi(\xi)\int_{S^{n-2}\cap\{ A(\theta)>3\delta\} }d\theta \int_{(R\ep)^\frac12}^r\frac{dp}{p^2}\int_{-p (A(\theta)-2\delta)}^{-\delta p}\(1+O\(\frac{1}{R\delta}\)\)\frac{dt}{(t^2+1)^\frac{n+1}{2}}\\&
\quad+\int_{\{|\xi|>\frac{R\delta}{2}\}}d\xi \,\dot\phi(\xi) \int_{S^{n-2}\cap\{ A(\theta)>3\delta\} }d\theta \int_{(R\ep)^\frac12}^r\frac{dp}{p^2}\int_{-p (A(\theta)-2\delta)}^{-\delta p}(\ldots)\,dt.
  \end{aligned}
  \end{equation}
   The main contribution in $\frac{1}{|\ln\ep|}\int_\R\dot\phi(\xi)J_1^3(\xi)d\xi$ comes from the integral of 1 in \eqref{eq:split-main-error}. Indeed, 
since $|t|\leq p(A(\theta)-2\delta)\leq Cr$ implies
   $$\frac{1}{\left(t^2+1\right)^\frac{n+1}{2}}=1+O(r),$$ we can write
  \begin{align*}
& \frac{1}{|\ln\ep|}\ \int_{S^{n-2}\cap\{ A(\theta)>3\delta\} }d\theta \int_{(R\ep)^\frac12}^r\frac{dp}{p^2}\int_{-p(A(\theta)-2\delta)}^{-p\delta}\frac{dt}{\left(t^2+1\right)^\frac{n+1}{2}}\\&\quad
=(1+O(r))\frac{1}{|\ln\ep|}\ \int_{S^{n-2}\cap\{ A(\theta)>3\delta\} }d\theta \int_{(R\ep)^\frac12}^r\frac{1}{p^2}p(A(\theta)-3\delta)\,dp\\&\quad
=(1+O(r))\frac{\ln r-\frac12\ln R-\frac12\ln\ep}{|\ln\ep|}\int_{S^{n-2}\cap\{  A(\theta)>3\delta\}}(A(\theta)-3\delta)\,d\theta\\&\quad
=\frac12\int_{S^{n-2}\cap\{  A(\theta)>3\delta\}}A(\theta)\,d\theta+o_\ep(1)+O(\delta)+O(r).
   \end{align*}
With this and recalling \eqref{phi-infyinftyaeptermxiestbis}, we infer that
\begin{equation}\label{maincontr12appendix} 
\begin{split}
&\frac{1}{|\ln\ep|} \int_{-\frac{R\delta}{2}}^{\frac{R\delta}{2}}d\xi\,\dot\phi(\xi)  \int_{S^{n-2}\cap\{  A(\theta)>3\delta\} }d\theta \int_{(R\ep)^\frac12}^r\frac{dp}{p^2}\int_{-p(A(\theta)-2\delta)}^{-p\delta}
\frac{dt}{\left(t^2+1\right)^\frac{n+1}{2}}
\\&\quad=\frac12\int_{S^{n-2}\cap\{ A(\theta)>3\delta\}}A(\theta)\,d\theta+o_\ep(1)+O(\delta)+O(r)+O\(\frac{1}{R\delta}\).
\end{split}
\end{equation}
Next, we look at the error terms in \eqref{eq:split-main-error}. First, note that
\begin{align*}
&\frac{1}{|\ln\ep|} \int_{S^{n-2}\cap\{  A(\theta)>3\delta\} }d\theta \int_{(R\ep)^\frac12}^r\frac{dp}{p^2}\int_{-p (A(\theta)-2\delta)}^{-\delta p}\frac{dt}{\left(t^2+1\right)^\frac{n+1}{2}}\\&\quad
\leq\frac{1}{|\ln\ep|} \int_{S^{n-2}\cap\{  A(\theta)>3\delta\} }d\theta \int_{(R\ep)^\frac12}^r\frac{1}{p^2}p( A(\theta)-3\delta)\,dp\\&\quad
\leq \frac{1}{|\ln\ep|}  \int_{S^{n-2}\cap\{  A(\theta)>3\delta\} }d\theta \int_{(R\ep)^\frac12}^r\frac{dp}{p}\\&\quad
\leq   C\frac{\ln r-\frac12\ln R-\frac12\ln \ep}{|\ln\ep|}
\leq C.
\end{align*}
With this, we estimate
\begin{equation}\label{J3error1} \begin{split}
&\left|\frac{1}{|\ln\ep|}  \int_{-\frac{R\delta}{2}}^{\frac{R\delta}{2}}d\xi \dot\phi(\xi) \int_{S^{n-2}\cap\{  A(\theta)>3\delta\} }d\theta \int_{(R\ep)^\frac12}^r\frac{dp}{p^2}\int_{-p (A(\theta)-2\delta)}^{-\delta p}O\(\frac{1}{R\delta}\)\frac{dt}{\left(t^2+1\right)^\frac{n+1}{2}}\right|\\&\leq O\(\frac{1}{R\delta}\),
\end{split}\end{equation}
and similarly, using that $0 < \phi < 1$ and \eqref{phi-infyinftyaeptermxiest}, 
\begin{equation}\label{J3error2}
\begin{aligned}0&\leq\frac{1}{|\ln\ep|}\int_{\{|\xi|>\frac{R\delta}{2}\}}d\xi\,\dot\phi(\xi) \int_{S^{n-2}\cap\{  A(\theta)>3\delta\} }d\theta \int_{(R\ep)^\frac12}^r\frac{dp}{p^2}\int_{-p (A(\theta)-2\delta)}^{-\delta p}(\ldots)\,dt\\&
\leq \frac{2}{|\ln\ep|}\int_{\{|\xi|>\frac{R\delta}{2}\}}d\xi\,\dot\phi(\xi)\int_{S^{n-2}\cap\{  A(\theta)>3\delta\} }d\theta \int_{(R\ep)^\frac12}^r\frac{dp}{p^2}\int_{-p (A(\theta)-2\delta)}^{-\delta p}\frac{dt}{\left(t^2+1\right)^\frac{n+1}{2}}
\\&
\leq  O\(\frac{1}{R\delta}\).
\end{aligned}
\end{equation}
Choosing 
\begin{equation}\label{eq:chooseR}
R=\delta^{-2},
\end{equation}
 from \eqref{eq:split-main-error}, 
 \eqref{maincontr12appendix}, \eqref{J3error1} and \eqref{J3error2}, estimate \eqref{eq:J13} follows. 
 
We next show that
\begin{equation}\label{eq:J1k}
\frac{1}{|\ln \ep|} \int_{\R} \dot{\phi}(\xi) J_1^k(\xi) \, d\xi = O(\delta) \	\quad \hbox{for}~k=1,2,4,5. 
\end{equation}
First, using that $0<\phi<1$, we get 
\begin{align*}
0&\leq \int_{\R}\dot\phi(\xi)J_1^2(\xi)\,d\xi
\leq 2\int_\R d\xi\,\dot\phi(\xi)\int_{S^{n-2}\cap\{  A(\theta)>3\delta\}}d\theta \int_{(R\ep)^\frac12}^r\frac{dp}{p^2}\int_{-\delta p} ^{\delta p }dt\\&
 \leq  2 \int_\R d\xi\,\dot\phi(\xi) \int_{(R\ep)^\frac12}^r\frac{dp}{p^2}2\delta p
\leq  C\delta|\ln\ep|,
\end{align*}
from which it follows that \eqref{eq:J1k} holds for $k=2$. 
The estimate for $k=4$ is similar. 

Regarding \eqref{eq:J1k} for $k=5$, we use \eqref{1=t^2termysmall} and the monotonicity of $\phi$ to estimate 
\begin{align*}
0&\leq\int_{S^{n-2}\cap\{A(\theta)>3\delta\} }d\theta \int_{(R\ep)^\frac12}^r\frac{dp}{p^2}\int_{-1}^{-p (A(\theta)+2\delta)}(\ldots)\,dt\\&
\leq\int_{S^{n-2}\cap\{A(\theta)>3\delta\} }d\theta \int_{(R\ep)^\frac12}^r\frac{dp}{p^2}\\&\qquad\times\int_{-1}^{-p (A(\theta)+2\delta)}
\left \{ \phi\(\xi + \frac{tp}{\ep}+\frac{p^2}{\ep}(A(\theta)+C_0r)\) - \phi\( \xi + \frac{tp}{\ep}\)\right\} dt\\&
 =\int_{S^{n-2}\cap\{A(\theta)>3\delta\} } d\theta  \int_{(R\ep)^\frac12}^r\frac{dp}{p^2}\int_{-1}^{-p(A(\theta)+2\delta)}dt\\&\qquad\times\int_0^1
 \dot\phi\(\xi + \frac{tp}{\ep}+\tau \frac{p^2}{\ep}(A(\theta)+C_0r) \)\frac{p^2(A(\theta)+C_0r)}{\ep}d\tau\\&
\leq \frac{C}{\ep}\int_{S^{n-2}\cap\{A(\theta)>3\delta\} } d\theta  \int_{(R\ep)^\frac12}^rdp\int_0^1d\tau \\&\qquad\times\int_{-1}^{-p(A(\theta)+2\delta)}\partial_t\left[
\phi\(\xi +\frac{tp}{\ep}+\tau \frac{p^2}{\ep}(A(\theta)+C_0r)\)\right]\frac{\ep}{p}\,dt\\&
 =C\int_{S^{n-2}\cap\{A(\theta)>3\delta\} }d\theta  \int_{(R\ep)^\frac12}^r \frac{dp}{p}\int_0^1\left\{ \phi\(\xi +\frac{p^2}{\ep}[ -A(\theta)-2\delta+\tau (A(\theta)+C_0r)]\)\right.\\&
\hspace{2.75in} \left. -\phi\(\xi -\frac{p}{\ep}+\tau\frac{p^2}{\ep}(A(\theta)+C_0r)\)\right\}d\tau. 
\end{align*}
Now, if
$$A(\theta)\geq 0, \quad (R\ep)^\frac{1}{2}\leq p\leq r, \quad 0\leq \tau\leq 1,\quad |\xi|\leq \frac{R\delta }{2},$$ 
and $\delta$ as in \eqref{deltaC_0-ine}, then
\begin{align*}\xi +\frac{p^2}{\ep}[ -A(\theta)-2\delta+\tau (A(\theta)+C_0r)]&\leq \xi-\frac{\delta p^2}{\ep}
\leq\frac{R\delta}{2} -R\delta
= -\frac{R\delta }{2}, \end{align*}
and  for $r$ and $\ep$  sufficiently small,
$$\xi  -\frac{p}{\ep}+\tau\frac{p^2}{\ep}(A(\theta)+C_0r)\leq \xi-\frac{p}{2\ep}\leq  -\frac{R\delta }{2}.$$
Recalling \eqref{eq:chooseR}, by \eqref{eq:asymptotics for phi}, we get
\[
\phi\(\xi +\frac{p^2}{\ep}[ -A(\theta)-2\delta+\tau (A(\theta)+C_0r)]\),\, \phi\(\xi -\frac{p}{\ep}+\tau\frac{p^2}{\ep}(A(\theta)+C_0r)\)\leq C\delta.
\]
The computations above yield 
\begin{align*}
&0\leq \int_{-\frac{R\delta}{2}}^{\frac{R\delta}{2}}
 d\xi\,\dot\phi(\xi)\int_{S^{n-2}\cap\{ A(\theta)>3\delta\} }d\theta \int_{(R\ep)^\frac12}^r\frac{dp}{p^2}\int_{-1}^{-p(A(\theta)+2\delta)}(\ldots)\,dt
\leq C\delta|\ln\ep|. 
\end{align*}
On the other hand, estimating as above and by \eqref{phi-infyinftyaeptermxiest} and \eqref{eq:chooseR},
\begin{align*}
0&\leq \int_{\{|\xi|>\frac{R\delta}{2}\}}d\xi\,\dot\phi(\xi)\int_{S^{n-2}\cap\{ A(\theta)>3\delta\} }d\theta \int_{(R\ep)^\frac12}^r\frac{dp}{p^2}\int_{-1}^{-p(A(\theta)+2\delta)}
(\ldots)\,dt
\\&\leq C\int_{\{|\xi|>\frac{R\delta}{2}\}}d\xi\,\dot\phi(\xi)\int_{S^{n-2}}d\theta \int_{(R\ep)^\frac{1}{2}}^r\frac{dp}{p}\int_0^1 d\tau\\&
 \leq C\delta|\ln\ep|.
\end{align*}
Together, we arrive at \eqref{eq:J1k} for $k=5$.
Similar computations yield \eqref{eq:J1k} for $k=1$. 

Combining \eqref{eq:J13} and \eqref{eq:J1k}, we obtain \eqref{J1}. 

 \medskip
 
 \noindent
\underline{\bf Step 1b. Estimating $\frac{1}{|\ln \ep|} \int_{\R} \dot{\phi}(\xi) J_2(\xi) \, d \xi$}.
We will show that
\begin{equation}\label{J2}
\frac{1}{|\ln \ep|} \int_{\R} \dot{\phi}(\xi) J_2(\xi) \, d \xi
	= \frac12\int_{S^{n-2}\cap\{ A(\theta)>3\delta\} }A(\theta)\,d\theta+o_\ep(1)+O(\delta)+O(r).
\end{equation}
We first write
\begin{align*}
J_2(\xi)
&=\int_{S^{n-2}\cap\{ A(\theta)>3\delta\} }d\theta \int_\ep^{(R\ep)^{\frac{1}{2}}}\frac{dp}{p^2}\int_{\delta \frac{\ep}{p}}^1(\ldots)\,dt\\&
\quad +\int_{S^{n-2}\cap\{ A(\theta)>3\delta\} }d\theta \int_\ep^{(R\ep)^{\frac{1}{2}}}\frac{dp}{p^2}\int_{-\delta \frac{\ep}{p}}^{\delta \frac{\ep}{p}}(\ldots)\,dt\\&
 \quad+\int_{S^{n-2}\cap\{ A(\theta)>3\delta\} }d\theta\int_\ep^{(R\ep)^{\frac{1}{2}}}\frac{dp}{p^2}\int_{-1}^{-\delta \frac{\ep}{p}}(\ldots)\,dt\\
 &=: J_2^1(\xi) + J_2^2(\xi) + J_2^3(\xi).  
  \end{align*}
Here the main contribution comes from $J_2^1$ and $J_2^3$. Indeed, we will show 
\begin{equation}\label{J_1estimatefinal}
\frac{1}{|\ln\ep|}\int_{\R}\dot\phi(\xi)J_2^1(\xi)\,d\xi=\frac14\int_{S^{n-2}\cap\{ A(\theta)>3\delta\} }A(\theta)\,d\theta+o_\ep(1)+O(r)+O(\delta), 
\end{equation}
and 
\begin{equation}\label{J_5estimatefinal}
\frac{1}{|\ln\ep|}\int_{\R}\dot\phi(\xi)J_2^3(\xi)\,d\xi=\frac14\int_{S^{n-2}\cap\{ A(\theta)>3\delta\} }A(\theta)\,d\theta+o_\ep(1)+O(r)+O(\delta). 
\end{equation}
Beginning with $J_2^3$, we split, for $R_0\geq K> 4$ to be chosen, 
\begin{equation}\label{eq:J1split-shortrho}\begin{split}
J_2^3(\xi)& = \int_{S^{n-2}\cap\{ A(\theta)>3\delta\} }d\theta \int_{\ep}^{R_0\ep}\frac{dp}{p^2}\int_{-1}^{-\delta \frac{ \ep}{p}} (\ldots)\,dt\\&
\quad+ \int_{S^{n-2}\cap\{ A(\theta)>3\delta\} }d\theta \int_{R_0\ep}^{(R\ep)^\frac12}\frac{dp}{p^2}\int_{-1}^{-K\frac{\ep}{p}}(\ldots)\,dt\\&
\quad+ \int_{S^{n-2}\cap\{ A(\theta)>3\delta\} }d\theta \int_{R_0\ep}^{(R\ep)^\frac12}\frac{dp}{p^2}\int_{-K\frac{\ep}{p}}^{-\delta\frac{\ep}{p}}(\ldots)\,dt.
\end{split}
\end{equation}
Regarding the bounds of integration over $t$, note that $K\frac{\ep}{p}\leq 1$ if $K\le R_0$ and $p\geq R_0\ep$. 
For the  first integral on the right-hand side of \eqref{eq:J1split-shortrho},  we estimate 
\begin{align*}
0&\leq \int_{S^{n-2}\cap\{ A(\theta)>3\delta\} }d\theta \int_{\ep}^{R_0\ep}\frac{dp}{p^2}\\& \qquad \times\int_{-1}^{-\delta \frac{\ep}{p}} 
\left\{ \phi\(\xi + \frac{p}{\ep}\left(t+p b(\theta,t, r)\right)\) - \phi\( \xi + \frac{tp}{\ep}\)\right\} \frac{dt}{\left(t^2+1\right)^\frac{n+1}{2}}\\&
\leq C\int_{S^{n-2}\cap\{ A(\theta)>3\delta\} }d\theta \int_{\ep}^{R_0\ep}\frac{dp}{p^2}\int_{-1}^{-\delta \frac{ \ep}{p}} \frac{p^2}{\ep}b(\theta,t,r)\,dt\\&
\leq \frac{C}{\ep}\int_{\ep}^{R_0\ep}\, dp\leq CR_0.
\end{align*}
It follows that
\begin{equation}\label{eq:J1split-shortrho:term1}
0\leq \frac{1}{|\ln\ep|}\int_\R d\xi\,\dot\phi(\xi)\int_{S^{n-2}\cap\{ A(\theta)>3\delta\} }d\theta \int_{\ep}^{R_0\ep}\frac{dp}{p^2}\int_{-1}^{-\delta \frac{ \ep}{p}} (\ldots)\,dt\leq\frac{CR_0}{|\ln\ep|}.
\end{equation}
For the second integral in \eqref{eq:J1split-shortrho}, by  \eqref{1=t^2termysmall} and the monotonicity of $\phi$,  we have
\begin{align*}
0&
\leq \int_{S^{n-2}\cap\{ A(\theta)>3\delta\} }d\theta \int_{R_0\ep}^{(R\ep)^\frac12}\frac{dp}{p^2}\\&\qquad\times\int_{-1}^{-K\frac{\ep}{p}}
\left \{ \phi\(\xi + \frac{p}{\ep}\left(t+p b(\theta,t, r)\right)\) - \phi\( \xi + \frac{tp}{\ep}\)\right\} \frac{dt}{\left(t^2+1\right)^\frac{n+1}{2}}\\&
\leq \int_{S^{n-2}\cap\{ A(\theta)>3\delta\} }d\theta \int_{R_0\ep}^{(R\ep)^\frac12}\frac{dp}{p^2}\\&\qquad\times\int_{-1}^{-K\frac{\ep}{p}}\left \{ \phi\(\xi + \frac{tp }{\ep}+\frac{p^2}{\ep}(A(\theta)+C_0r)\) - \phi\( \xi + \frac{tp}{\ep}\)\right\} \frac{dt}{\left(t^2+1\right)^\frac{n+1}{2}}\\&
 \leq\int_{S^{n-2}\cap\{ A(\theta)>3\delta\} }d\theta \int_{R_0\ep}^{(R\ep)^\frac12}\frac{dp}{p^2}\int_{-1}^{-K\frac{\ep}{p}}dt\\&\qquad\times \int_0^1
 \dot\phi\(\xi  +\frac{tp }{\ep}+\tau\frac{p^2}{\ep}(A(\theta)+C_0r)\) 
\frac{p^2}{\ep}(A(\theta)+C_0r)\,d\tau
\\&
\leq\frac{C}{\ep} \int_{S^{n-2}}d\theta \int_{R_0\ep}^{(R\ep)^\frac{1}{2}}dp\int_0^1 d\tau\int_{-1}^{-K\frac{\ep}{p}}\partial_t
\left[\phi\(\xi + \frac{p t}{\ep}+\tau\frac{p^2}{\ep}(A(\theta)+C_0r)\)\right]\frac{\ep}{p}\,dt\\&
 =C\int_{S^{n-2}}d\theta \int_{R_0\ep}^{(R\ep)^\frac{1}{2}}\frac{dp}{p}\\&\qquad\times\int_0^1 \left\{\phi\(\xi -K+\tau\frac{p^2}{\ep}(A(\theta)+C_0r)\) -\phi\(\xi -\frac{p}{\ep}+\tau\frac{p^2}{\ep}(A(\theta)+C_0r)\)\right\}d\tau.
\end{align*}
If 
$$R_0\ep\leq p\leq (R\ep)^\frac{1}{2}, \quad 0\leq \tau\leq 1,\quad |\xi|\leq \frac{K}{2},$$ then we have that,  for some $\tilde C\geq1$,
$$\xi-K+\tau\frac{p^2}{\ep}(A(\theta)+C_0r)\leq   -\frac{K}{2}+ \tilde CR,$$
$$\xi-\frac{p}{\ep}+\tau\frac{p^2}{\ep}(A(\theta)+C_0r)\leq 
\frac{K}{2} -R_0+ \tilde CR. $$
 Let $R_0$ and $K$ be such that 
\begin{equation}\label{R0Kdefstep1aii}  4\leq 4 \tilde C R= K\leq R_0.\end{equation}
Then, 
\[
\xi-K+\tau\frac{p^2}{\ep}(A(\theta)+C_0r),\, \,\xi-\frac{p}{\ep}+\tau\frac{p^2}{\ep}(A(\theta)+C_0r)\leq -\frac{K}{4}
\]
and by  \eqref{eq:asymptotics for phi}, 
$$\phi\(\xi -K+\tau\frac{p^2}{\ep}(A(\theta)+C_0r)\),\,\phi\(\xi -\frac{p}{\ep}+\tau\frac{p^2}{\ep}(A(\theta)+C_0r)\)\leq \frac{C}{K}.$$
This implies that, for $|\xi|\leq K/2$, 
\begin{align*}
&\int_{S^{n-2}}d\theta \int_{R_0\ep}^{(R\ep)^\frac{1}{2}}\frac{dp}{p}\int_0^1 \left\{\phi\(\xi -K
\frac{p^2}{\ep}(A(\theta)+C_0r)\)-
\phi\(\xi -\frac{p}{\ep}+\frac{p^2}{\ep}(A(\theta)+C_0r)\)\right\}d\tau\\&
\quad \leq \frac{C}{K} \int_{R_0\ep}^{(R\ep)^\frac{1}{2}}\frac{dp}{p}\
 \leq  \frac{C|\ln\ep|}{K}. 
\end{align*}
The computations above yield 
\begin{align*}
0&\leq \int_{-\frac{K}{2}}^{\frac{K}{2}} d\xi\,\dot\phi(\xi)\int_{S^{n-2}\cap\{ A(\theta)>3\delta\} }d\theta \int_{R_0\ep}^{(R\ep)^\frac12}\frac{dp}{p^2}\int_{-1}^{-K\frac{\ep}{p}}(\ldots)\,dt
\leq \frac{C|\ln\ep|}{K}. 
\end{align*}
On the other hand, estimating as above but using that $0 < \phi < 1$ and \eqref{phi-infyinftyaeptermxiest}, we obtain
\begin{align*}
0&\leq \int_{\{|\xi|>\frac{K}{2}\}}d\xi\,\dot\phi(\xi)\int_{S^{n-2}\cap\{ A(\theta)>3\delta\} }d\theta \int_{R_0\ep}^{(R\ep)^\frac12}\frac{dp}{p^2}\int_{-1}^{-K\frac{\ep}{p}}
(\ldots)\,dt
\\&\leq C\int_{\{|\xi|>\frac{K}{2}\}}d\xi\,\dot\phi(\xi)\int_{S^{n-2}}d\theta \int_{R_0\ep}^{(R\ep)^\frac{1}{2}}\frac{dp}{p}\int_0^1 d\tau
\leq \frac{C|\ln\ep|}{K}.
\end{align*}
We conclude that 
\begin{equation}\label{eq:J1split-shortrho:term2}
0\leq \frac{1}{|\ln\ep|}\int_\R d\xi\,\dot\phi(\xi)\int_{S^{n-2}\cap\{ A(\theta)>3\delta\} }d\theta \int_{R_0\ep}^{(R\ep)^\frac12}\frac{dp}{p^2}\int_{-1}^{-K\frac{p}{\ep}}(\ldots)\,dt\leq\frac{C}{K}.
\end{equation}
We next estimate the third term on the right-hand side of \eqref{eq:J1split-shortrho}. We first notice that, if $|t|\leq K\frac{\ep}{p}\leq1$ and $p\geq R_0\ep$, then 

$$\frac{1}{(t^2+1)^\frac{n+1}{2}}=1+O\left(\frac{K}{R_0}\right).$$
We set 
\begin{equation}\label{R0defstep1aii}R_0=K^2,\end{equation}
then
\begin{align*}
&\int_{S^{n-2}\cap\{ A(\theta)>3\delta\} }d\theta \int_{R_0\ep}^{(R\ep)^\frac12}\frac{dp}{p^2}\int_{-K\frac{\ep}{p}}^{-\delta\frac{\ep}{p}}
\left \{ \phi\(\xi + \frac{p}{\ep}\left(t+p b(\theta,t, r)\right)\) - \phi\( \xi + \frac{tp}{\ep}\)\right\} \frac{dt}{\left(t^2+1\right)^\frac{n+1}{2}}\\&
\quad=\int_{S^{n-2}\cap\{ A(\theta)>3\delta\} }d\theta \int_{R_0\ep}^{(R\ep)^\frac12}\frac{dp}{p^2}\int_{-K\frac{\ep}{p}}^{-\delta\frac{\ep}{p}}
\(\ldots\)\left(1+ O(K^{-1})\right)dt.
\end{align*}
Using again \eqref{1=t^2termysmall} and the monotonicity of $\phi$, we get
\begin{align*}
&\int_{S^{n-2}\cap\{ A(\theta)>3\delta\} }d\theta \int_{R_0\ep}^{(R\ep)^\frac12}\frac{dp}{p^2}\int_{-K\frac{\ep}{p}}^{-\delta\frac{\ep}{p}}
\left \{ \phi\(\xi + \frac{p}{\ep}\left(t+p b(\theta,t, r)\right)\) - \phi\( \xi + \frac{tp}{\ep}\)\right\} dt\\&
\quad\leq \int_{S^{n-2}\cap\{ A(\theta)>3\delta\} }d\theta \int_{R_0\ep}^{(R\ep)^\frac12}\frac{dp}{p^2}\int_{-K\frac{\ep}{p}}^{-\delta\frac{\ep}{p}}
\left \{ \phi\(\xi + \frac{tp}{\ep}+\frac{p^2}{\ep}(A(\theta)+C_0 r)\) - \phi\( \xi + \frac{tp}{\ep}\)\right\} dt\\&
\quad=\int_{S^{n-2}\cap\{ A(\theta)>3\delta\} }d\theta \int_{R_0\ep}^{(R\ep)^\frac12}\frac{dp}{p^2}\\&\qquad\quad\times\int_{-K\frac{\ep}{p}}^{-\delta\frac{\ep}{p}}dt\int_0^1
\dot\phi\(\xi + \frac{tp}{\ep}+\tau\frac{p^2}{\ep}(A(\theta)+C_0 r)\) 
\frac{p^2}{\ep}(A(\theta)+C_0 r)\,d\tau\\&
\quad=\int_{S^{n-2}\cap\{ A(\theta)>3\delta\} }d\theta \int_{R_0\ep}^{(R\ep)^\frac12}\frac{dp}{p}\int_0^1 d\tau\\&
\quad\qquad\times\int_{-K\frac{\ep}{p}}^{-\delta\frac{\ep}{p}}\partial_t\left[\phi\(\xi + \frac{tp}{\ep}+\tau\frac{p^2}{\ep}(A(\theta)+C_0 r)\)\right](A(\theta)+C_0r)\,dt\\&
\quad= \int_{S^{n-2}\cap\{ A(\theta)>3\delta\} }d\theta \int_{R_0\ep}^{(R\ep)^\frac12}\frac{dp}{p}\int_0^1 d\tau\int_{-K\frac{\ep}{p}}^{-\delta\frac{\ep}{p}}
\partial_t\left[\phi\(\xi + \frac{tp}{\ep}+\tau\frac{p^2}{\ep}(A(\theta)+C_0 r)\)\right]A(\theta)\,dt\\&
\quad\qquad +O(r|\ln\ep|). 
\end{align*}
Similarly, 
\begin{align*}
&\int_{S^{n-2}\cap\{ A(\theta)>3\delta\} }d\theta \int_{R_0\ep}^{(R\ep)^\frac12}\frac{dp}{p^2}\int_{-K\frac{\ep}{p}}^{-\delta\frac{\ep}{p}}
\left \{ \phi\(\xi + \frac{p}{\ep}\left(t+p b(\theta,t, r)\right)\) - \phi\( \xi + \frac{tp}{\ep}\)\right\} dt\\&
\quad\geq 
\int_{S^{n-2}\cap\{ A(\theta)>3\delta\} }d\theta \int_{R_0\ep}^{(R\ep)^\frac12}\frac{dp}{p}\int_0^1 d\tau\int_{-K\frac{\ep}{p}}^{-\delta\frac{\ep}{p}}
\partial_t\left[\phi\(\xi + \frac{tp}{\ep}+\tau\frac{p^2}{\ep}(A(\theta)-C_0 r)\)\right]A(\theta)\,dt \\&
\quad\quad+O(r|\ln\ep|). 
\end{align*}
For the main terms in the integrals above, we 
write
\begin{align*}
&\int_{S^{n-2}\cap\{ A(\theta)>3\delta\} }d\theta  A(\theta)\int_{R_0\ep}^{(R\ep)^\frac12}\frac{dp}{p}\int_0^1 d\tau\int_{-K\frac{\ep}{p}}^{-\delta\frac{\ep}{p}}
\partial_t\left[\phi\(\xi +\frac{tp}{\ep}+\tau\frac{p^2}{\ep}(A(\theta)\pm C_0 r)\)\right]dt\\&
\quad=\int_{S^{n-2}\cap\{ A(\theta)>3\delta\} }d\theta A(\theta)\int_{R_0\ep}^{(R\ep)^\frac12}\frac{dp}{p}\\&
\qquad\times\int_0^1 \left\{\phi\(\xi  -\delta +\tau\frac{p^2}{\ep}(A(\theta)\pm C_0 r)\)-\phi\(\xi  -K+\tau\frac{p^2}{\ep}(A(\theta)\pm C_0 r)\)\right\}d\tau\\&
\quad =\int_{S^{n-2}\cap\{ A(\theta)>3\delta\} }d\theta A(\theta) \int_{R_0\ep}^{(R\ep)^\frac12}
\left\{\phi\(\xi  -\delta\)-\phi\(\xi  -K\)+O\left(\frac{ p^2}{\ep}\right)\right\}\frac{dp}{p}. 
\end{align*}
As above, we find that
$$\int_\R d\xi\,\dot\phi(\xi)\int_{S^{n-2}\cap\{ A(\theta)>3\delta\} }d\theta A(\theta) \int_{R_0\ep}^{(R\ep)^\frac12}\phi\(\xi  -K\)\frac{dp}{p}\leq \frac{C|\ln\ep|}{K}. 
$$
Moreover, regarding the error term, 
$$\int_\R d\xi\,\dot\phi(\xi)\int_{S^{n-2}\cap\{ A(\theta)>3\delta\} }d\theta  A(\theta)\int_{R_0\ep}^{(R\ep)^\frac12}\frac{dp}{p}\frac{p^2}{\ep}\leq CR. 
$$
The main contribution comes from the following integral 
\begin{align*}
&\int_\R d\xi\,\dot\phi(\xi)\int_{S^{n-2}\cap\{ A(\theta)>3\delta\} }d\theta\, A(\theta) \int_{R_0\ep}^{(R\ep)^\frac12}
\phi\(\xi  -\delta\)\frac{dp}{p}\\&
\quad =\int_\R d\xi\,\dot\phi(\xi)\int_{S^{n-2}\cap\{ A(\theta)>3\delta\} }d\theta\,  A(\theta) \int_{R_0\ep}^{(R\ep)^\frac12}
\phi\(\xi\)\frac{dp}{p}+O(\delta|\ln\ep|)\\&
\quad=\int_\R \frac12\frac{d}{d\xi}\left(\phi^2(\xi)\right)d\xi\int_{S^{n-2}\cap\{ A(\theta)>3\delta\} }d\theta\, A(\theta) \int_{R_0\ep}^{(R\ep)^\frac12}\frac{dp}{p}+O(\delta|\ln\ep|)\\&
\quad=\frac12\(\frac12|\ln\ep|+\frac12\ln R-\ln R_0\)\int_{S^{n-2}\cap\{ A(\theta)>3\delta\} } A(\theta)\,d\theta+O(\delta|\ln\ep|),
\end{align*}
where we used that $\phi(\infty)=1$ and $\phi(-\infty)=0$.
Putting it all together, we get 
\begin{equation}\label{eq:J1split-shortrho:term3}
\begin{split}
&\frac{1}{|\ln\ep|}\int_\R d\xi\,\dot\phi(\xi)\int_{S^{n-2}\cap\{ A(\theta)>3\delta\} }d\theta \int_{R_0\ep}^{(R\ep)^\frac12}\frac{dp}{p^2}\int_{-K\frac{\ep}{p}}^{-\delta\frac{\ep}{p}}(\ldots)\,dt\\&\quad=\frac14\int_{S^{n-2}\cap\{ A(\theta)>3\delta\} }A(\theta)\,d\theta+o_\ep(1)+O(r)+O(K^{-1})+O(\delta). 
\end{split}\end{equation}
Recalling \eqref{eq:chooseR}, \eqref{R0Kdefstep1aii} and \eqref{R0defstep1aii}, from 
\eqref{eq:J1split-shortrho:term1}, \eqref{eq:J1split-shortrho:term2} and \eqref{eq:J1split-shortrho:term3}, 
we get \eqref{J_5estimatefinal}. 

We now check that the estimate for $J_2^1$ in \eqref{J_1estimatefinal} holds. For $R_0,\,K$ as in \eqref{R0Kdefstep1aii} and \eqref{R0defstep1aii}, we write
\begin{equation}\label{eq:J5split-shortrho}\begin{split}
J_2^1(\xi)
 &= \int_{S^{n-2}\cap\{ A(\theta)>3\delta\} }d\theta \int_{\ep}^{R_0\ep}\frac{dp}{p^2}\int_{\delta \frac{\ep}{p}}^{1}(\ldots)\,dt\\&
 \quad+ \int_{S^{n-2}\cap\{ A(\theta)>3\delta\} }d\theta \int_{R_0\ep}^{(R\ep)^\frac12}\frac{dp}{p^2}\int_{\delta\frac{\ep}{p}}^{K\frac{\ep}{p}}(\ldots)\,dt\\&
\quad+ \int_{S^{n-2}\cap\{ A(\theta)>3\delta\} }d\theta \int_{R_0\ep}^{(R\ep)^\frac12}\frac{dp}{p^2}\int_{K\frac{\ep}{p}}^{1}(\ldots)\,dt.
\end{split}
\end{equation}
Similar computations as for the estimates \eqref{eq:J1split-shortrho:term1} and \eqref{eq:J1split-shortrho:term2}  yield 
\begin{equation}\label{eq:J5split-shortrho:term1}
0\leq\frac{1}{|\ln\ep|}\int_\R d\xi\dot\phi(\xi)\int_{S^{n-2}\cap\{ A(\theta)>3\delta\} }d\theta \int_{\ep}^{R_0\ep}\frac{dp}{p^2}\int_{\delta \frac{\ep}{p}}^{1}(\ldots)\,dt\leq\frac{CR_0}{|\ln\ep|},
\end{equation}
and 
\begin{equation}\label{eq:J5split-shortrho:term2}
0\leq\frac{1}{|\ln\ep|}\int_\R d\xi\dot\phi(\xi)\int_{S^{n-2}\cap\{ A(\theta)>3\delta\} }d\theta \int_{R_0\ep}^{(R\ep)^\frac12}\frac{dp}{p^2}\int_{K\frac{p}{\ep}}^1(\ldots)\,dt\leq\frac{C}{K}. 
\end{equation}
The  second term on the right-hand side of \eqref{eq:J5split-shortrho} is similar to \eqref{eq:J1split-shortrho:term3}. Indeed, as above,
\begin{equation}\label{J5thirdtermfirstest}\begin{split}
& \int_{S^{n-2}\cap\{ A(\theta)>3\delta\} }d\theta \int_{R_0\ep}^{(R\ep)^\frac12}\frac{dp}{p^2}\int_{\delta\frac{\ep}{p}}^{K\frac{\ep}{p}}(\ldots)\,dt\\&
\quad=\int_{S^{n-2}\cap\{ A(\theta)>3\delta\} }d\theta A(\theta) \int_{R_0\ep}^{(R\ep)^\frac12}
\left\{\phi\(\xi  +K\)-\phi\(\xi  +\delta\)\right\}\frac{dp}{p}\\
&\qquad+O(R)+O(K^{-1}|\ln\ep|)+O(r|\ln\ep|). 
\end{split}
\end{equation}
By  \eqref{eq:asymptotics for phi}, for $|\xi|\leq \frac{K}{2}$,
$$\phi\(\xi  +K\)=1+O\left(\frac{1}{K}\right).$$
Therefore, 
\begin{align*}
&\int_{-\frac{K}{2}}^\frac{K}{2}d\xi\,\dot\phi(\xi)\int_{S^{n-2}\cap\{ A(\theta)>3\delta\} }d\theta \int_{R_0\ep}^{(R\ep)^\frac12}\frac{dp}{p^2}\int_{\delta\frac{\ep}{p}}^{K\frac{\ep}{p}}(\ldots)\,dt\\&
\quad = \int_{-\frac{K}{2}}^\frac{K}{2}d\xi\,\dot\phi(\xi)\int_{S^{n-2}\cap\{ A(\theta)>3\delta\} }d\theta \,A(\theta) \int_{R_0\ep}^{(R\ep)^\frac12}
\left\{1-\phi\(\xi\)+O(K^{-1})+O(\delta)\right\}\frac{dp}{p}\\&
\quad\quad+O(R)+O(K^{-1}|\ln\ep|))+O(r|\ln\ep|)\\&
\quad= \int_{-\frac{K}{2}}^\frac{K}{2}d\xi\left\{\dot\phi(\xi)-\frac12\frac{d}{d\xi}\left(\phi(\xi)\right)^2\right\}\int_{S^{n-2}\cap\{ A(\theta)>3\delta\} }d\theta A(\theta) \int_{R_0\ep}^{(R\ep)^\frac12}\frac{dp}{p}\\&
\quad\quad+O(K^{-1}|\ln\ep|)+O(\delta|\ln\ep|)+O(R)+O(r|\ln\ep|)\\&
\quad=\left[\phi\(\frac{K}{2}\)-\phi\(-\frac{K}{2}\)-\frac12\(\phi^2\(\frac{K}{2}\)-\phi^2\(-\frac{K}{2}\)\)\right]\int_{S^{n-2}\cap\{ A(\theta)>3\delta\} }d\theta  A(\theta)\int_{R_0\ep}^{(R\ep)^\frac12}\frac{dp}{p}\\&
\quad\quad+O(K^{-1}|\ln\ep|)+O(\delta|\ln\ep|)+O(R)+O(r|\ln\ep|).
\end{align*}
Since, again by  \eqref{eq:asymptotics for phi}, 
$$\phi\(\frac{K}{2}\)-\phi\(-\frac{K}{2}\)-\frac12\(\phi^2\(\frac{K}{2}\)-\phi^2\(-\frac{K}{2}\)\)=\frac12+O(K^{-1}),$$
we get
\begin{align*}
&\int_{-\frac{K}{2}}^\frac{K}{2}d\xi\,\dot\phi(\xi) \int_{S^{n-2}\cap\{ A(\theta)>3\delta\} }d\theta \int_{R_0\ep}^{(R\ep)^\frac12}\frac{dp}{p^2}\int_{\delta\frac{\ep}{p}}^{K\frac{\ep}{p}}(\ldots)\,dt\\&
\quad=\left(\frac12+O(K^{-1})\right)\int_{S^{n-2}\cap\{ A(\theta)>3\delta\} }d\theta A(\theta) \int_{R_0\ep}^{(R\ep)^\frac12}\frac{dp}{p}\\&
\quad\quad+O(K^{-1}|\ln\ep|)+O(\delta|\ln\ep|)+O(R)+O(r|\ln\ep|)\\&
\quad=\frac12\int_{S^{n-2}\cap\{ A(\theta)>3\delta\} }d\theta A(\theta)\int_{R_0\ep}^{(R\ep)^\frac12}\frac{dp}{p}+O(K^{-1}|\ln\ep|)+O(\delta|\ln\ep|)+O(R)+O(r|\ln\ep|)\\&
\quad=\frac12\(\frac12|\ln\ep|+\frac12\ln R-\ln R_0\)\int_{S^{n-2}\cap\{ A(\theta)>3\delta\} }A(\theta)\,d\theta\\&
\quad\quad+O(K^{-1}|\ln\ep|)+O(\delta|\ln\ep|)+O(R)+O(r|\ln\ep|).
\end{align*}
On the other hand,  by \eqref{phi-infyinftyaeptermxiest} and \eqref{J5thirdtermfirstest}, we get
\begin{align*}
& \int_{\{|\xi|>\frac{K}{2}\}}d\xi\,\dot\phi(\xi) \int_{S^{n-2}\cap\{ A(\theta)>3\delta\} }d\theta \int_{R_0\ep}^{(R\ep)^\frac12}\frac{dp}{p^2}\int_{\delta\frac{\ep}{p}}^{K\frac{\ep}{p}}(\ldots)\,dt\\&
\quad=\int_{\{|\xi|>\frac{K}{2}\}}d\xi\,\dot\phi(\xi)\left[\int_{S^{n-2}\cap\{ A(\theta)>3\delta\} }d\theta A(\theta) \int_{R_0\ep}^{(R\ep)^\frac12}
\left\{\phi\(\xi  +K\)-\phi\(\xi  +\delta\)\right\}\frac{dp}{p}\right.\\&
\qquad\left. + \,O(R)+O(K^{-1}|\ln\ep|)+O(r|\ln\ep|)\right]\\&
\quad \leq C|\ln\ep| \int_{\{|\xi|>\frac{K}{2}\}}\dot\phi(\xi)\,d\xi
=O(K^{-1}|\ln\ep|). 
\end{align*}
The previous two estimates give
\begin{equation}\label{eq:J5split-shortrho:term3}
\begin{split}
&\frac{1}{|\ln\ep|}\int_\R d\xi\,\dot\phi(\xi)\int_{S^{n-2}\cap\{ A(\theta)>3\delta\} }d\theta \int_{R_0\ep}^{(R\ep)^\frac12}\frac{dp}{p^2}\int_{\delta\frac{\ep}{p}}^{K\frac{\ep}{p}}(\ldots)\,dt\\&\quad=\frac14\int_{S^{n-2}\cap\{ A(\theta)>3\delta\} }A(\theta)\,d\theta+o_\ep(1)+O(r)+O(K^{-1})+O(\delta). 
\end{split}\end{equation}
From \eqref{eq:J5split-shortrho}, 
\eqref{eq:J5split-shortrho:term1},   \eqref{eq:J5split-shortrho:term2}, \eqref{eq:J5split-shortrho:term3} and 
recalling  \eqref{eq:chooseR}, \eqref{R0Kdefstep1aii} and \eqref{R0defstep1aii}, we get \eqref{J_1estimatefinal}.

Lastly, we will show that
\begin{equation}\label{eq:J23}
\frac{1}{|\ln \ep|} \int_{\R} \dot{\phi}(\xi)J_2^2(\xi) \, d \xi = O(\delta).
\end{equation}
Indeed, 
\begin{align*}
 0\leq J_2^2(\xi) &= \int_{S^{n-2}\cap\{A(\theta)>3\delta\} }d\theta \int_{\ep}^{(R\ep)^\frac12}\frac{dp}{p^2}\int_{-\delta\frac{\ep}{p}}^{\delta\frac{\ep}{p}}\\
 &\qquad\left\{ \phi\(\xi + \frac{p}{\ep}\left(t+p b(\theta,t, r)\right)\) - \phi\( \xi + \frac{tp}{\ep}\)\right\} \frac{dt}{\left(t^2+1\right)^\frac{n+1}{2}}\\&
 \leq C \int_{S^{n-2}\cap\{A(\theta)>3\delta\} }d\theta \int_{\ep}^{(R\ep)^\frac12}\frac{dp}{p^2}\int_{-\delta\frac{\ep}{p}}^{\delta\frac{\ep}{p}}\frac{p^2}{\ep} b(\theta,r,t) \, dt \\
 &\leq \frac{C}{\ep}\int_{S^{n-2}\cap\{ A(\theta)>3\delta\} }d\theta \int_{\ep}^{(R\ep)^\frac12}\delta\frac{\ep}{p}\, dp\leq C\delta|\ln\ep|,
\end{align*}
from which we obtain \eqref{eq:J23}.

Therefore, combining \eqref{J_1estimatefinal}, \eqref{J_5estimatefinal}  and \eqref{eq:J23}, we have \eqref{J2}. 

\medskip
 
\noindent
\underline{\bf Completion of Step 1}.
Recall \eqref{Ixisplitint} and \eqref{J1J2splitofI_1}.  With \eqref{J1} and \eqref{J2}, 
we can finally write 
\begin{equation}\label{eq:I_1^1}
 \frac{1}{|\ln \ep|} \int_{\R} \dot{\phi}(\xi) I_1^1(\xi) \, d \xi
	=\int_{S^{n-2}\cap\{ A(\theta)>3\delta\} }A(\theta)\,d\theta+o_\ep(1) +O(\delta)+ O(r).
\end{equation}
In the same way, we obtain
\begin{equation}\label{eq:I_1^2}
 \frac{1}{|\ln \ep|} \int_{\R} \dot{\phi}(\xi) I_1^2(\xi) \, d \xi
	=\int_{S^{n-2}\cap\{ A(\theta)<-3\delta\} }A(\theta)\,d\theta+o_\ep(1) +O(\delta)+ O(r).
\end{equation}
Finally, let us show that 
\begin{equation}\label{eq:I_1^3}
 \frac{1}{|\ln \ep|} \int_{\R} \dot{\phi}(\xi) I_1^3(\xi) \, d \xi
	=o_\ep(1) +o_\delta(1)+ O(r).
\end{equation}
If one of the eigenvalues $\lambda_i$ is different than zero, then $\mathcal{H}^{n-2}(\{\theta\in S^{n-2}\,|\,A(\theta)=0\})=0$.
In particular, $\mathcal{H}^{n-2}(\{\theta\in S^{n-2}\,|\,|A(\theta)|<3\delta\})=o_\delta(1).$
Therefore, integrating in $t$ as before, 
\begin{align*}
\left| \int_{\R} \dot{\phi}(\xi) I_1^3(\xi) \, d \xi\right|&\leq C \int_\R d\xi\,\dot\phi(\xi)\int_{S^{n-2}\cap\{ |A(\theta)|<3\delta\} }d\theta \int_{\ep}^r\frac{dp}{p}\\&
\leq C|\ln\ep|\int_{S^{n-2}\cap\{ |A(\theta)|<3\delta\} }d\theta
=|\ln\ep|o_\delta(1),
\end{align*} which implies \eqref{eq:I_1^3}. 

If instead, $\lambda_i=0$ for all $i=1,\ldots,n-1$, then $A(\theta)\equiv0$ and $S^{n-2}\cap\{|A(\theta)|<3\delta\}=S^{n-2}$. In this case, we write, 
 for $\delta$ and $R$ as in 
\eqref{deltaC_0-ine} and \eqref{eq:chooseR}, 
\begin{align*}
 I_1^3(\xi)&=
\int_{S^{n-2} }d\theta \int_{(R\ep)^\frac12}^r\frac{dp}{p^2}\int_{-1}^{-2\delta p}(\ldots)\,dt+\int_{S^{n-2}\ }d\theta \int_{(R\ep)^\frac12}^r\frac{dp}{p^2}\int_{-2\delta p}^{2\delta p}(\ldots)\,dt\\&
 \quad+\int_{S^{n-2}}d\theta \int_{(R\ep)^\frac12}^r\frac{dp}{p^2}\int_{2\delta p}^{1}(\ldots)\,dt\\&
 \quad +\int_{S^{n-2} }d\theta \int_{\ep}^{(R\ep)^\frac12}\frac{dp}{p^2}\int_{-1}^{-\delta \frac{\ep}{p}}(\ldots)\,dt+\int_{S^{n-2}\ }d\theta \int_{\ep}^{(R\ep)^\frac12}\frac{dp}{p^2}
 \int_{-\delta \frac{\ep}{p}}^{\delta \frac{\ep}{p}}(\ldots)\,dt\\&
 \quad+\int_{S^{n-2}}d\theta \int_{\ep}^{(R\ep)^\frac12}\frac{dp}{p^2}\int_{\delta \frac{\ep}{p}}^{1}(\ldots)\,dt.
  \end{align*}
As for the estimates of $\int_{\R}\dot\phi(\xi)(J_1^1(\xi)+J_1^2(\xi)+J_1^5(\xi))\,d\xi$ (recall \eqref{eq:J1k}), 
\begin{align*}
\frac{1}{|\ln\ep|}\int_{\R}d\xi\,\dot\phi(\xi)\int_{S^{n-2} }d\theta\int_{(R\ep)^\frac12}^r\frac{dp}{p^2}\int_{-1}^{-2\delta p}(\ldots)\,dt
&=O(\delta),
 \end{align*}
 \begin{align*}
\frac{1}{|\ln\ep|}\int_{\R}d\xi\,\dot\phi(\xi)\int_{S^{n-2} }d\theta\int_{(R\ep)^\frac12}^r\frac{dp}{p^2}\int_{-2\delta p}^{2\delta p}(\ldots)\,dt
&=O(\delta),
 \end{align*}
and 
\begin{align*}
\frac{1}{|\ln\ep|}\int_{\R}d\xi\,\dot\phi(\xi)\int_{S^{n-2} }d\theta\int_{(R\ep)^\frac12}^r\frac{dp}{p^2}\int_{2\delta p}^{1}(\ldots)\,dt
&=O(\delta).
 \end{align*}
As for the estimates of
 $\int_{\R}\dot\phi(\xi)(J_2^1(\xi)+J_2^2(\xi)+J_2^3(\xi))\,d\xi$ (recall  \eqref{J_1estimatefinal}, \eqref{J_5estimatefinal} and \eqref{eq:J23}), 
\begin{align*}
\frac{1}{|\ln\ep|}\int_{\R}d\xi\,\dot\phi(\xi)\int_{S^{n-2} }d\theta \int_{\ep}^{(R\ep)^\frac12}\frac{dp}{p^2}\int_{-1}^{-\delta \frac{\ep}{p}}(\ldots)\,dt
&=\frac14\int_{S^{n-2}}A(\theta)d\theta+o_\ep(1)+O(r)+O(\delta)\\&
=o_\ep(1)+O(r)+O(\delta),
 \end{align*}
 \begin{align*}
\frac{1}{|\ln\ep|}\int_{\R}d\xi\,\dot\phi(\xi)\int_{S^{n-2} }d\theta \int_{\ep}^{(R\ep)^\frac12}\frac{dp}{p^2}\int_{-\delta \frac{\ep}{p}}^{\delta \frac{\ep}{p}}(\ldots)\,dt
&=O(\delta),
 \end{align*}
and 
\begin{align*}
\frac{1}{|\ln\ep|}\int_{\R}d\xi\,\dot\phi(\xi)\int_{S^{n-2} }d\theta \int_{\ep}^{(R\ep)^\frac12}\frac{dp}{p^2}\int_{\delta \frac{\ep}{p}}^1(\ldots)\,dt&=\frac14\int_{S^{n-2}}A(\theta)d\theta+o_\ep(1)+O(r)+O(\delta)\\&
=o_\ep(1)+O(r)+O(\delta).
\end{align*}
Estimate \eqref{eq:I_1^3} then follows. 

From \eqref{eq:I_1^1}, \eqref{eq:I_1^2} and \eqref{eq:I_1^3}, and recalling \eqref{deltaC_0-ine}, we choose  $\delta=C_0r$ to    finally get  \eqref{eq:I_1-final}. 

\medskip

 \noindent
\underline{\bf Step 2. Estimating $\frac{1}{|\ln \ep|} \int_{\R} \dot{\phi}(\xi) I_2(\xi) \, d \xi$}.
We will show that
\begin{equation}\label{I_2est} 
\frac{1}{|\ln\ep|}\int_{\R}\dot\phi(\xi)I_2(\xi)\,d\xi = o_r(1).
\end{equation}
Recalling \eqref{bthetatrde}, we see that for $|t|>1$, there is $C_1>0$ such that 
\begin{equation*} 
A(\theta)-C_1rt^2\leq b(\theta,t,r)\leq A(\theta)+C_1rt^2. \end{equation*}
Then, for $C_1$ as above and $R>2$, to be determined,  by the monotonicity of $\phi$, 
\begin{align*}
I_2(\xi)&\leq 
 \int_{\ep}^r\frac{dp}{p^2}\int_{S^{n-2}}d\theta\int_{\left\{|t|>1\right\}}
 \left\{ \phi\(\xi + \frac{tp}{\ep}+ \frac{p^2}{\ep} (A(\theta)+C_1rt^2)\) - \phi\( \xi + \frac{tp}{\ep}\)\right\} \frac{dt}{\left(t^2+1\right)^\frac{n+1}{2}}\\
 &=\int_{\ep}^r\frac{dp}{p^2}\int_{S^{n-2}}d\theta\int_{\left\{|t|>\frac{1}{2\sqrt{rRC_1p}}\right\}}(\ldots)\,dt
+ \int_{\ep}^r\frac{dp}{p^2}\int_{S^{n-2}}d\theta\int_{\left\{1<|t|<\frac{1}{2\sqrt{rRC_1p}}\right\}}(\ldots)\,dt\\&=:I_2^1(\xi)+I_2^2(\xi).
\end{align*}
We first estimate
\begin{align*}I_2^1(\xi)&\leq 2 \int_{\ep}^r\frac{dp}{p^2}\int_{S^{n-2}}d\theta\int_{\left\{|t|>\frac{1}{2\sqrt{rRC_1p}}\right\}}\frac{dt}{|t|^{n+1}}\\
&\leq C\int_{\ep}^r(rR)^\frac{n}{2}p^{\frac{n}{2}-2}\,dp\leq C(rR)^\frac{n}{2}\int_{\ep}^r\frac{dp}{p} 
\leq C(rR)^\frac{n}{2}|\ln\ep|,
\end{align*}
so that  
\begin{equation}\label{I_2^1estimate}\frac{1}{|\ln\ep|}\int_\R\dot\phi(\xi)I_2^1(\xi)\,d\xi\leq C(rR)^\frac{n}{2}.
\end{equation}
Next, let us estimate $\int_\R\dot\phi(\xi)I_2^2(\xi)\,d\xi$. 
If 
\begin{equation*}\label{xitpRassumestI_2^1}|\xi|\leq\frac{|t|p}{\ep}-\frac{p}{R\ep},\quad 0\leq p\leq r<1, \quad 1\leq |t|\leq \frac{1}{2\sqrt{rRC_1p}},\quad R>2,\quad 0\leq \tau\leq 1, 
\end{equation*} and $R,\,r$ are such that 
\begin{equation}\label{xitpRassumestI_2^1bis} r|A(\theta)|\leq \frac{1}{4R}, \end{equation} 
then 
\begin{align*}\left|\xi + \frac{tp}{\ep}+ \tau \frac{p^2}{\ep} (A(\theta)+C_1rt^2)\right|&\geq  \frac{|t|p}{\ep}-|\xi|- \frac{p}{4R\ep}-\frac{p}{4R\ep}
\geq \frac{p}{2R\ep}.
\end{align*}
Therefore, by \eqref{eq:asymptotics for phi dot},  for some $\tau\in(0,1)$, 
\begin{equation*}
\begin{split}
& \phi\(\xi + \frac{tp}{\ep}+ \frac{p^2}{\ep} (A(\theta)+C_1rt^2)\) - \phi\( \xi + \frac{tp}{\ep}\)\\&\quad=\dot\phi\(\xi + \frac{tp}{\ep}+\tau \frac{p^2}{\ep} (A(\theta)+C_1rt^2)\)\frac{p^2}{\ep} (A(\theta)+C_1rt^2)\\&
\quad\leq \frac{C}{\left|\xi + \frac{tp}{\ep}+ \tau \frac{p^2}{\ep} (A(\theta)+C_1rt^2)\right|^2} \frac{p^2}{\ep}(1+rt^2) \leq 
C(1+rt^2)R^2\ep,
\end{split}
\end{equation*}
from which we find that
 \begin{align*}
& \int_{\ep}^r\frac{dp}{p^2}\int_{S^{n-2}}d\theta\int_{\left\{1<|t|<\frac{1}{2\sqrt{rRC_1p}}\right\}}\frac{dt}{\left(t^2+1\right)^\frac{n+1}{2}}
\int_{-\frac{|t|p}{\ep}+\frac{p}{R\ep}}^{\frac{|t|p}{\ep}-\frac{p}{R\ep}}
\dot\phi(\xi)
\\
&\qquad\times \left\{ \phi\(\xi + \frac{tp}{\ep}+ \frac{p^2}{\ep} (A(\theta)+C_1rt^2)\) - \phi\( \xi + \frac{tp}{\ep}\) \right\}d\xi
 \\&
 \quad\leq C R^2\ep\int_{\R}d\xi \,\dot\phi(\xi)\int_{S^{n-2}}d\theta
\int_{\ep}^r\frac{dp}{p^2}  \int_{\left\{1<|t|<\frac{1}{2\sqrt{rRC_1p}}\right\}} \frac{1+rt^2}{|t|^{n+1}}\,dt\\&
\quad\leq C R^2\ep
\int_{\ep}^r\frac{dp}{p^2}  \int_{\left\{1<|t|<\frac{1}{2\sqrt{rRC_1\ep}}\right\}}\( \frac{1}{|t|^{n+1}}+\frac{r}{|t|}\)\,dt\\&
\quad\leq C R^2\(1+r|\ln\ep|\).
 \end{align*}
Consequently,
\begin{equation}\label{intI_2^2shortxi}\begin{split}
&\frac{1}{|\ln\ep|}\int_{\ep}^r\frac{dp}{p^2}\int_{S^{n-2}}d\theta\int_{\left\{1<|t|<\frac{1}{2\sqrt{rRC_1p}}\right\}} \frac{dt}{\left(t^2+1\right)^\frac{n+1}{2}}
\int_{\{|\xi|<\frac{|t|p }{\ep}-\frac{p}{R\ep}\}}\dot\phi(\xi)
\\&\qquad \times \left\{ \phi\(\xi + \frac{tp}{\ep}+ \frac{p^2}{\ep} (A(\theta)+C_1rt^2)\)  - \phi\( \xi + \frac{tp}{\ep}\) \right\} d\xi
\\&\quad\leq CR^2\(\frac{1}{|\ln\ep|}+r\). 
\end{split}
\end{equation}
Next,  again by \eqref{eq:asymptotics for phi dot}, for   $|t|>1$, $R>2$ and for some  $\tau_1,\,\tau_2\in(-1,1)$, 
 \begin{align*}
 &\int_{\left\{\frac{|t|p}{\ep}-\frac{p}{R\ep}<|\xi|<\frac{|t|p}{\ep}+\frac{p}{R\ep}\right\}}\dot\phi(\xi)\,d\xi\\
 &\qquad=\(\phi\(\frac{|t|p}{\ep}+\frac{p}{R\ep}\)-\phi\(\frac{|t|p}{\ep}-\frac{p}{R\ep}\)\)+\(\phi\(-\frac{|t|p}{\ep}+\frac{p}{R\ep}\)-\phi\(-\frac{|t|p}{\ep}-\frac{p}{R\ep}\)\) \\ &\qquad
=\dot\phi\(\frac{|t|p}{\ep}+\tau_1\frac{p}{R\ep}\)\frac{2p}{R\ep}+\dot\phi\(-\frac{|t|p}{\ep}+\tau_2\frac{p}{R\ep}\)\frac{2p}{R\ep}\\ &\qquad
\leq C\(\frac{\ep}{tp}\)^2\frac{p}{R\ep}
=C\frac{\ep}{Rp t^2}.
\end{align*}
Therefore,
 \begin{align*}
 &\int_{\ep}^r\frac{dp}{p^2}\int_{S^{n-2}}d\theta\int_{\left\{1<|t|<\frac{1}{2\sqrt{rRC_1p}}\right\}} \frac{dt}{\left(t^2+1\right)^\frac{n+1}{2}}\int_{\left\{\frac{|t|p}{\ep}-\frac{p}{R\ep}<|\xi|<\frac{|t|p}{\ep}+\frac{p}{R\ep}\right\}}\dot\phi(\xi)
\\&\qquad \times \left\{ \phi\(\xi + \frac{tp}{\ep}+ \frac{p^2}{\ep} (A(\theta)+C_1rt^2)\)  - \phi\( \xi + \frac{tp}{\ep}\) \right\} d\xi
\\&\quad\leq  C\int_{S^{n-2}}d\theta\int_{|t|>1}dt\,\frac{1+rt^2}{|t|^{n+1}}\int_{\ep}^r \frac{dp}{p^2}\frac{p^2}{\ep}\int_{\left\{\frac{|t|p}{\ep}-\frac{p}{R\ep}<|\xi|<\frac{|t|p}{\ep}+\frac{p}{R\ep}\right\}}\dot\phi(\xi)\,d\xi\\&
\quad\leq \frac{C}{R}\int_{\ep}^r \frac{dp}{p}\int_{|t|>1}\frac{dt}{|t|^{n+1}}\\&
\quad\leq \frac{C|\ln\ep|}{R}.
 \end{align*}
 Consequently,
  \begin{equation}\label{I_2^2estimateximiddle}\begin{split}
  &\frac{1}{|\ln\ep|}\int_{\ep}^r\frac{dp}{p^2}\int_{S^{n-2}}d\theta\int_{\left\{1<|t|<\frac{1}{2\sqrt{rRC_1p}}\right\}} \frac{dt}{\left(t^2+1\right)^\frac{n+1}{2}}\int_{\left\{\frac{|t|p}{\ep}-\frac{p}{R\ep}<|\xi|<\frac{|t|p}{\ep}+\frac{p}{R\ep}\right\}}\dot\phi(\xi)
\\&\qquad\quad\quad \left\{ \phi\(\xi + \frac{tp}{\ep}+ \frac{p^2}{\ep} (A(\theta)+C_1rt^2)\) - \phi\( \xi + \frac{tp}{\ep}\) \right\} d\xi
\leq  \frac{C}{R}.
 \end{split}
   \end{equation}
   Finally, if 
   $$|\xi|\geq \frac{|t|p}{\ep}+\frac{p}{R\ep},\quad 0\leq p\leq r, \quad 1\leq |t|\leq \frac{1}{2\sqrt{rRC_1p}},\quad R>2,\quad 0\leq \tau\leq 1,$$ 
   and \eqref{xitpRassumestI_2^1bis} holds true, then 
  \begin{align*}\left|\xi + \frac{tp}{\ep}+ \tau\frac{p^2}{\ep} (A(\theta)+C_1rt^2)\right|&\geq |\xi|-\frac{|t|p}{\ep}- \frac{p^2}{\ep}|A(\theta)|-\frac{p^2}{\ep}C_1rt^2
\geq  \frac{p}{2R\ep},
\end{align*}
 and as before, by \eqref{eq:asymptotics for phi dot},
\begin{align*}
\phi\(\xi + \frac{tp}{\ep}+ \frac{p^2}{\ep} (A(\theta)+C_1rt^2)\) - \phi\( \xi + \frac{tp}{\ep}\)
\leq C(1+rt^2)R^2\ep.
\end{align*}
Therefore,
 \begin{align*}
& \int_{\ep}^r\frac{dp}{p^2}\int_{S^{n-2}}d\theta\int_{\left\{1<|t|<\frac{1}{2\sqrt{rRC_1p}}\right\}} \frac{dt}{\left(t^2+1\right)^\frac{n+1}{2}}\int_{\{|\xi|>\frac{|t|p}{\ep}+\frac{p}{R\ep}\}}\dot\phi(\xi)\\
&\quad\quad \times \left\{\phi\(\xi + \frac{tp}{\ep}+ \frac{p^2}{\ep} (A(\theta)+C_1rt^2)\)  - \phi\( \xi + \frac{tp}{\ep}\) \right\}d\xi
 \\&
 \quad\leq C R^2\ep
\int_{\ep}^r\frac{dp}{p^2}  \int_{\left\{1<|t|<\frac{1}{2\sqrt{rRC_1\ep }}\right\}} \frac{1+rt^2}{|t|^{n+1}}\,dt\\&
\quad\leq C R^2\(1+r|\ln\ep|\),
 \end{align*}
 and 
 \begin{equation}\label{intI_2^2largexi}\begin{split}
&\frac{1}{|\ln\ep|}\int_{\ep}^r\frac{dp}{p^2}\int_{S^{n-2}}d\theta\int_{\left\{1<|t|<\frac{1}{2\sqrt{rRC_1p}}\right\}}\frac{dt}{\left(t^2+1\right)^\frac{n+1}{2}} \int_{\{|\xi|>\frac{|t|p}{\ep}+\frac{p}{R\ep}\}}\dot\phi(\xi)
\\& \quad\quad
\times \left\{ \phi\(\xi + \frac{tp}{\ep}+ \frac{p^2}{\ep} (A(\theta)+C_1rt^2)\) - \phi\( \xi + \frac{tp}{\ep}\) \right\} d\xi
\leq CR^2\(\frac{1}{|\ln\ep|}+r\). 
\end{split}
\end{equation}
From \eqref{intI_2^2shortxi}, \eqref{I_2^2estimateximiddle} and \eqref{intI_2^2largexi},  choosing $R=r^{-\frac13}$, we have that \eqref{xitpRassumestI_2^1bis} holds true for $r$ small enough, and 
\begin{equation*}\frac{1}{|\ln\ep|}\int_\R\dot\phi(\xi)I_2^2(\xi)\,d\xi \leq o_r(1).
\end{equation*}
Together with \eqref{I_2^1estimate}, this gives
\begin{equation*}
\frac{1}{|\ln\ep|}\int_{\R}\dot\phi(\xi)I_2(\xi)\,d\xi\leq o_r(1).
\end{equation*}
The  lower bound for $\int_{\R}\dot\phi(\xi)I_2(\xi)\,d\xi$ is obtained in a similar way. Estimate  
\eqref{I_2est}  follows.  

\medskip

 \noindent
\underline{\bf Conclusion}.
Recalling \eqref{Ixisplitintbo}, we combine \eqref{eq:I_1-final} and \eqref{I_2est} to finally obtain
\begin{equation}\label{eq:I-final}
\frac{1}{|\ln \ep|} \int_{\R} \dot{\phi}(\xi)I(\xi) \, d\xi
	= \int_{S^{n-2} } A(\theta ) \, d \theta +o_\ep(1)+ o_r(1),
\end{equation}
where $o_\ep(1)$ depends on the parameter $r$. 
From \eqref{eq:aebarI-bis} and  \eqref{eq:I-final}, we first send $\ep\to0$ and then $r\to0$ to arrive at
\[
\lim_{\ep \to 0} \bar{a}_\ep(x)
	=   \int_{S^{n-2}} A(\theta) \, d \theta,
\]
uniformly in $Q_\rho$.
Recalling \eqref{eq:Atheta-MC}, this gives the desired result. 
 \qed
 
\section{Proof of Theorem \ref{lem:4} for $s \in (\frac12,1)$}\label{sec:s-large}

Throughout this section, we assume that $s  \in ( \frac12,1)$.

Fix $r \in(0,\frac{\rho}{2})$ and take $0 < \ep <r$. 
By Lemma \ref{lem:ae-change-of-variables}, we have
\begin{equation}\label{eq:aebarI-s}
\bar{a}_\ep(x)
= O\left(\frac{\ep^{2s-1}}{r^{2s}}\right) 
+ \ep^{2s-1} \int_{\R} G(\xi) \dot{\phi}(\xi) \, d \xi
\end{equation}
where
\begin{align*}
G(\xi)
&=\int_{0}^r \frac{dp}{p^{2s+1}} \int_{S^{n-2}} d \theta
	 \int_{-\frac{r}{p}}^{\frac{r}{p}}\frac{dt}{ (t^2 +1)^{\frac{n+2s}{2}}} \left\{ \phi\(\xi + \frac{p}{\ep}\(t+pb(\theta,t,r)\)\) - \phi\( \xi + \frac{tp}{\ep}\)\right\}.
\end{align*}
We start by splitting $G(\xi)$, for $R>0$ to be chosen, as
 \begin{equation}\label{eq:I-split-s}
 \begin{split}
 G(\xi)
 &=
 \int_{0}^{R\ep}\frac{dp}{p^{2s+1}}\int_{S^{n-2}}d\theta \int_{-\frac{r}{p}}^{\frac{r}{p}}(\ldots)\,dt
 	+\int_{R\ep}^{r}\frac{dp}{p^{2s+1}}\int_{S^{n-2}}d\theta \int_{-\frac{r}{p}}^{\frac{r}{p}}(\ldots)\,dt\\
	&=:I_1(\xi)+I_2(\xi),
 \end{split}
 \end{equation}
and estimate $\ep^{2s-1} \int_{\R} \dot{\phi}(\xi) I_1(\xi) \, d \xi$ and $\ep^{2s-1} \int_{\R} \dot{\phi}(\xi) I_2(\xi) \, d \xi$  separately. 

\medskip
 \noindent
\underline{\bf Step 1. Estimating $\ep^{2s-1} \int_{\R} \dot{\phi}(\xi) I_1(\xi) \, d \xi$}. 
We will show that
\begin{equation}\label{eq:I1new-final-s}
\ep^{2s-1} \int_{\R} \dot{\phi}(\xi) I_1(\xi) \, d \xi = 
c_1c_2\int_{S^{n-2}} A(\theta) \, d \theta
+ o_\ep(1) +o_r(1)
\end{equation}
where $c_1>0$ is defined in \eqref{finite-energy} and  $c_2>0$ depends only on $n$ and $s$ and is to be determined. 
Recalling \eqref{bthetatrde}, we write
\begin{align*}
&I_1(\xi)\\
&= \int_{0}^{R\ep} \frac{dp}{p^{2s+1}} \int_{S^{n-2}} d \theta
	 \int_{-\frac{r}{p}}^{\frac{r}{p}}
	 \left\{ \phi\(\xi + \frac{p}{\ep}\(t+pb(\theta,t,r)\)\) - \phi\( \xi + \frac{tp}{\ep}\)\right\}\frac{dt}{ (t^2 +1)^{\frac{n+2s}{2}}}\\&
	 = \int_{0}^{R\ep} \frac{dp}{p^{2s+1}} \int_{S^{n-2}} d \theta
	\int_{-\frac{r}{p}}^{\frac{r}{p}}\frac{dt}{ (t^2 +1)^{\frac{n+2s}{2}}}\\&\quad \times\int_0^1\dot \phi\(\xi + \frac{p}{\ep}\(t+\tau pb(\theta,t,r)\)\)\frac{p^2}{\ep}\(A(\theta) + O\left(r(1+t^2)\right)\)\,d\tau\\&
	  = \frac{1}{\ep} \int_{0}^{R\ep} \frac{dp}{p^{2s-1}} \int_{S^{n-2}}  d \theta\, A(\theta)
	\int_{-\frac{r}{p}}^{\frac{r}{p}}\frac{dt}{ (t^2 +1)^{\frac{n+2s}{2}}}\int_0^1
	 \dot \phi\(\xi + \frac{p}{\ep}\(t+\tau pb(\theta,t,r)\)\)\,d\tau\\&
	  \quad + \frac{O(r)}{\ep} \int_{0}^{R\ep} \frac{dp}{p^{2s-1}} \int_{S^{n-2}} d \theta
	\int_{-\frac{r}{p}}^{\frac{r}{p}}\frac{dt}{ (t^2 +1)^{\frac{n+2s-2}{2}}}\\&
	 = \frac{1}{\ep} \int_{0}^{R\ep} \frac{dp}{p^{2s-1}} \int_{S^{n-2}} d \theta\, A(\theta)
	\int_{-\frac{r}{p}}^{\frac{r}{p}}\frac{dt}{ (t^2 +1)^{\frac{n+2s}{2}}}\int_0^1
	 \(\dot \phi\(\xi + \frac{pt}{\ep}\) + O\left(\tau \frac{p^2}{\ep}b(\theta,t,r)\right)\)\,d\tau\\&
	  \quad + \frac{O(r)}{\ep} \int_{0}^{R\ep} \frac{dp}{p^{2s-1}} \int_{S^{n-2}} d \theta
	\int_{-\frac{r}{p}}^{\frac{r}{p}}\frac{dt}{ (t^2 +1)^{\frac{n+2s-2}{2}}}\\&
	 =I_1^1(\xi)+E_1(\xi)+E_2(\xi),
\end{align*}
where we take
\begin{equation*}
	I_1^1(\xi)=\frac{1}{\ep}\int_{0}^{R\ep} \frac{dp}{p^{2s-1}} \int_{S^{n-2}}d \theta\, A(\theta)
 \int_{-\frac{r}{p}}^{\frac{r}{p}}\dot \phi\(\xi + \frac{pt}{\ep}\)\frac{dt}{ (t^2 +1)^{\frac{n+2s}{2}}},
 \end{equation*}
 \begin{equation}\label{E_1-s}
	E_1(\xi)=\frac{1}{\ep^2}\int_{0}^{R\ep} dp\,p^{3-2s}
 \int_{-\frac{r}{p}}^{\frac{r}{p}}
 O\(1+r(1+t^2)\)\frac{dt}{ (t^2 +1)^{\frac{n+2s}{2}}},
 \end{equation}
 and 
 \begin{equation}\label{E_2-s}
	E_2(\xi)=\frac{O(r)}{\ep}\int_{0}^{R\ep} \frac{dp}{p^{2s-1}}
 \int_{-\frac{r}{p}}^{\frac{r}{p}}\frac{dt}{ (t^2 +1)^{\frac{n+2s-2}{2}}}.
 \end{equation} 
Integrating by parts in $t$, we have
\begin{align*}
I_1^1(\xi)&
=\frac{1}{2\ep}\int_{0}^{R\ep} \frac{dp}{p^{2s-1}} \int_{S^{n-2}} d \theta\, A(\theta)
\int_{-\frac{r}{p}}^{\frac{r}{p}}\(\dot \phi\(\xi + \frac{pt}{\ep}\)+\dot \phi\(\xi -\frac{pt}{\ep}\)\)\frac{dt}{ (t^2 +1)^{\frac{n+2s}{2}}}\\&
 =\frac{1}{2}\int_{0}^{R\ep} \frac{dp}{p^{2s}} \int_{S^{n-2}} d \theta\, A(\theta)
 \int_{-\frac{r}{p}}^{\frac{r}{p}}\partial_t\left[\phi\(\xi + \frac{pt}{\ep}\)-\phi\(\xi - \frac{pt}{\ep}\)\right]\frac{dt}{ (t^2 +1)^{\frac{n+2s}{2}}}\\&
 =\frac{1}{2}\int_{0}^{R\ep} \frac{dp}{p^{2s}} \int_{S^{n-2}} d \theta\, A(\theta)\left[\(\phi\(\xi + \frac{pt}{\ep}\)-\phi\(\xi - \frac{pt}{\ep}\)\)\frac{1}{ (t^2 +1)^{\frac{n+2s}{2}}}\Big\vert_{t=-\frac{r}{p}}^{t=\frac{r}{p}}\right.\\&
 \quad \left.+(n+2s) \int_{-\frac{r}{p}}^{\frac{r}{p}}\(\phi\(\xi + \frac{pt}{\ep}\)-\phi\(\xi -\frac{pt}{\ep}\)\) \frac{t}{ (t^2 +1)^{\frac{n+2s+2}{2}}}\,dt\right]\\&
 =I_1^2(\xi)+E_3(\xi),
 \end{align*}
 where 
 \begin{equation*}
 \begin{split}
 I_1^2(\xi)=\frac{n+2s}{2} \int_{0}^{R\ep} \frac{dp}{p^{2s}} \int_{S^{n-2}} d \theta\, A(\theta)
\int_{-\frac{r}{p}}^{\frac{r}{p}}\(\phi\(\xi + \frac{pt}{\ep}\)-\phi\(\xi - \frac{pt}{\ep}\)\)\frac{t}{ (t^2 +1)^{\frac{n+2s+2}{2}}}dt
\end{split}
  \end{equation*}
  and 
\begin{equation}\label{E_3-s}
E_3(\xi)=\frac12\int_{0}^{R\ep} \frac{dp}{p^{2s}} \int_{S^{n-2}} d \theta\, A(\theta)\,\(\phi\(\xi + \frac{pt}{\ep}\)-\phi\(\xi - \frac{pt}{\ep}\)\)\frac{1}{ (t^2 +1)^{\frac{n+2s}{2}}}\Big\vert_{t=-\frac{r}{p}}^{t=\frac{r}{p}}.
  \end{equation}
Note that the integrals above are well defined as,
$$|I_1^2(\xi)|,\,|E_3(\xi)|\leq \frac{C}{\ep}\int_{0}^{R\ep} p^{1-2s}\,dp\leq CR^{2-2s}\ep^{1-2s}. $$
  
  \noindent 
  Therefore, integrating by parts with respect to $\xi$, 
  we get
  \begin{align*}
 \int_\R I_1^2(\xi)\dot\phi(\xi)\,d\xi
 &=\frac{n+2s}{2} \int_{0}^{R\ep} \frac{dp}{p^{2s}} \int_{S^{n-2}} d \theta\, A(\theta)
\int_{-\frac{r}{p}}^{\frac{r}{p}}dt\,\frac{t}{ (t^2 +1)^{\frac{n+2s+2}{2}}}\\
&\qquad\times\int_\R\dot \phi(\xi)\(\phi\(\xi + \frac{pt}{\ep}\)-\phi\(\xi - \frac{pt}{\ep}\)\)\,d\xi\\&
=-\frac{n+2s}{2} \int_{0}^{R\ep} \frac{dp}{p^{2s}} \int_{S^{n-2}} d \theta\, A(\theta)
\int_{-\frac{r}{p}}^{\frac{r}{p}}dt\,\frac{t}{ (t^2 +1)^{\frac{n+2s+2}{2}}}\\
&\qquad \times\int_\R\phi(\xi)\(\dot\phi\(\xi + \frac{pt}{\ep}\)-\dot\phi\(\xi - \frac{pt}{\ep}\)\)\,d\xi.
\end{align*}
Integrating by parts again in $t$, we find
 \begin{align*}
 \int_\R I_1^2(\xi)\dot\phi(\xi)\,d\xi
&=-\ep \frac{n+2s}{2} \int_{0}^{R\ep} \frac{dp}{p^{1+2s}} \int_{S^{n-2}}d \theta\, A(\theta) \int_\R d \xi\, \phi(\xi)\\
&\qquad\times\int_{-\frac{r}{p}}^{\frac{r}{p}}
\partial_t \left[\phi\(\xi + \frac{pt}{\ep}\)+\phi\(\xi - \frac{pt}{\ep}\)-2\phi(\xi)\right] \frac{t}{ (t^2 +1)^{\frac{n+2s+2}{2}}}\,dt\\
&=-\ep\frac{n+2s}{2} \int_{0}^{R\ep} \frac{dp}{p^{1+2s}} \int_{S^{n-2}} d \theta\, A(\theta)\int_\R d \xi\, \phi(\xi)\\
&\qquad\times
\Bigg[ \left(\phi\(\xi + \frac{pt}{\ep}\)+\phi\(\xi - \frac{pt}{\ep}\)-2\phi(\xi)\right) \frac{t}{(t^2+1)^{\frac{n+2s+2}{2}}}\Big|_{t=-\frac{r}{p}}^{t = \frac{r}{p}}\\
&\qquad - \int_{-\frac{r}{p}}^{\frac{r}{p}}
\(\phi\(\xi + \frac{pt}{\ep}\)+\phi\(\xi - \frac{pt}{\ep}\)-2\phi(\xi)\) \frac{d}{dt} \left[ \frac{t}{ (t^2 +1)^{\frac{n+2s+2}{2}}}\right]\,dt
\Bigg]\\
&= F + E_4,
\end{align*}
where 
\begin{equation*}
\begin{split}
F&=  \ep \frac{n+2s}{2} \int_{0}^{R\ep} \frac{dp}{p^{1+2s}} \int_{S^{n-2}} d \theta\, A(\theta)\int_\R d\xi\, \phi(\xi)\\&\qquad\times\int_{-\frac{r}{p}}^{\frac{r}{p}} dt
\(\phi\(\xi + \frac{pt}{\ep}\)+\phi\(\xi - \frac{pt}{\ep}\)-2\phi(\xi)\)\frac{d}{dt}\left[\frac{t}{ (t^2 +1)^{\frac{n+2s+2}{2}}} \right],
\end{split}
\end{equation*}
and 
\begin{equation}\label{E_4-s}
\begin{split}
E_4&=-\ep \frac{n+2s}{2} \int_{0}^{R\ep} \frac{dp}{p^{1+2s}} \int_{S^{n-2}} d \theta\, A(\theta)\int_\R d\xi\\&
\qquad \times \phi(\xi)
\(\phi\(\xi + \frac{pt}{\ep}\)+\phi\(\xi - \frac{pt}{\ep}\)-2\phi(\xi)\) \frac{t}{ (t^2 +1)^{\frac{n+2s+2}{2}}}\bigg\vert_{t=-\frac{r}{p}}^{t=\frac{r}{p}}.
\end{split}
\end{equation}
In $F$, we will make the change of variable $w=pt/\ep$ in $t$ and then $q =  p /(\ep|w|)$ in $p$. In this regard, it is helpful to first write
\begin{align*}
\frac{d}{dt} \left[ \frac{t}{(t^2+1)^{\frac{n+2s+2}{2}}}\right]
	&= \frac{1}{(t^2+1)^{\frac{n+2s+2}{2}}} - (n+2s+2) \frac{t^2}{(t^2+1)^{\frac{n+2s+4}{2}}}\\
	&= \frac{1}{ \left(\frac{\ep^2}{p^2} w^2+1 \right)^{\frac{n+2s+2}{2}}} - (n+2s+2) \frac{\frac{\ep^2}{p^2} w^2}{ \left(\frac{\ep^2}{p^2} w^2+1 \right)^{\frac{n+2s+4}{2}}}\\
	&= \frac{1}{ \left(q^{-2}+1 \right)^{\frac{n+2s+2}{2}}} - (n+2s+2) \frac{q^{-2}}{ \left(q^{-2}+1 \right)^{\frac{n+2s+4}{2}}}\\
	&=q^{n+2s+2} \left[\frac{1}{ \left(1+q^2 \right)^{\frac{n+2s+2}{2}}} - \frac{n+2s+2}{ \left(1+q^2 \right)^{\frac{n+2s+4}{2}}}\right].
\end{align*}
Therefore, making the aforementioned change of variables,
\begin{align*}
F&=  \ep \frac{n+2s}{2} \int_{0}^{R\ep} \frac{dp}{p^{1+2s}} \int_{S^{n-2}} d \theta\, A(\theta)\int_\R d\xi\, \phi(\xi)\int_{-\frac{r}{\ep}}^{\frac{r}{\ep}} \frac{\ep}{p} \, d w
\\&\qquad\times\(\phi\(\xi +w\)+\phi\(\xi +w\)-2\phi(\xi)\) \left[
\frac{1}{ \left(\frac{\ep^2}{p^2} w^2+1 \right)^{\frac{n+2s+2}{2}}} -  \frac{(n+2s+2)\frac{\ep^2}{p^2} w^2}{ \left(\frac{\ep^2}{p^2} w^2+1 \right)^{\frac{n+2s+4}{2}}}\right]\\
&= \ep^2\frac{n+2s}{2} \int_{S^{n-2}} d \theta\, A(\theta)\int_\R d\xi\, \phi(\xi)\int_{-\frac{r}{\ep}}^{\frac{r}{\ep}}
\, d w \int_{0}^{\frac{R}{|w|}}
\frac{dq}{(\ep |w|)^{1+2s}q^{2+2s}}
\\&\qquad \times\(\phi\(\xi +w\)+\phi\(\xi -w\)-2\phi(\xi)\)  q^{n+2s+2} \left[\frac{1}{ \left(1+q^2 \right)^{\frac{n+2s+2}{2}}} -  \frac{n+2s+2}{ \left(1+q^2 \right)^{\frac{n+2s+4}{2}}}\right]\\
&= \ep^{1-2s} \frac{n+2s}{2}   \int_{S^{n-2}} d \theta\, A(\theta)\int_\R d\xi\, \phi(\xi)\int_{-\frac{r}{\ep}}^{\frac{r}{\ep}}dw \, \frac{\phi\(\xi +w\)+\phi\(\xi -w\)-2\phi(\xi) }{|w|^{1+2s}} \\
&\qquad \times\int_{0}^{\frac{R}{|w|}} 
 q^{n} \left[\frac{1}{ \left(1+q^2 \right)^{\frac{n+2s+2}{2}}} -  \frac{n+2s+2}{ \left(1+q^2 \right)^{\frac{n+2s+4}{2}}}\right]\,dq.
\end{align*}
Note that,   by estimate \eqref{eq:asymptotics for phi dot} for $\ddot\phi$ and Taylor's Theorem, for $|w|<|\xi|/2$,
$$|\phi\(\xi +w\)+\phi\(\xi -w\)-2\phi(\xi)|\leq C\frac{w^2}{1+|\xi|^{2s+1}}.$$
Thus,
\begin{align*}
&\int_\R d\xi\,\phi(\xi)\int_{\R}\frac{|\phi\(\xi +w\)+\phi\(\xi -w\)-2\phi(\xi)|}{|w|^{1+2s}} \, d w\\&
\quad\leq C \int_\R d\xi\, \phi(\xi)\left[\int_{\{|w|<\frac{\max\{|\xi|,1\}}{2}\}}\frac{|w|^{1-2s}}{1+|\xi|^{2s+1}}\,dw+\int_{\{|w|>\frac{\max\{|\xi|,1\}}{2}\}}\frac{1}{|w|^{1+2s}}\,dw\right]\\&\quad
\leq C\int_\R \left(\frac{1}{1+|\xi|^{4s-1}}+\frac{1}{1+|\xi|^{2s}}\right) \,d\xi
\leq C.
\end{align*}
Moreover, for any $M>1$,
\begin{align*}
 &\int_\R d\xi\, \phi(\xi)\int_{\{|w|>M\}}\frac{|\phi\(\xi +w\)+\phi\(\xi -w\)-2\phi(\xi)|}{|w|^{1+2s}}\,dw\\&
 \quad =\int_{-M}^Md\xi\,(\ldots)+\int_{\{|\xi|>M\}}d\xi\, (\ldots)\\&
 \quad \leq C\int_{-M}^Md\xi\, \int_{\{|w|>M\}}\frac{dw}{|w|^{1+2s}}+C\int_{\{|\xi|>M\}}\left(\frac{1}{1+|\xi|^{4s-1}}+\frac{1}{1+|\xi|^{2s}}\right) \,d\xi\\&
 \quad  \leq\frac{C}{M^{2s-1}}+\frac{C}{M^{2(2s-1)}}\leq \frac{C}{M^{2s-1}}.
\end{align*}
Therefore, for  $c_1$  defined in \eqref{finite-energy}, we have
\begin{align*}
&-\frac12\int_\R d\xi\, \phi(\xi)\int_{-\frac{r}{\ep}}^{\frac{r}{\ep}}\frac{\phi\(\xi +w\)+\phi\(\xi -w\)-2\phi(\xi)} {|w|^{1+2s}}  \, d w\\&
=-\frac12\int_\R d\xi\, \phi(\xi)\int_{-R^\frac12}^{R^\frac12}\frac{\phi\(\xi +w\)+\phi\(\xi -w\)-2\phi(\xi)} {|w|^{1+2s}} \, d w+ O(R^{-(s-\frac12)})+O\(\(\frac{\ep}{r}\)^{2s-1}\)\\ 
&=\frac12\int_\R\int_\R \frac{(\phi\(\xi +w\)-\phi(\xi))^2}{|w|^{1+2s}}\,dw\,d\xi+  O(R^{-(s-\frac12)})+O\(\(\frac{\ep}{r}\)^{2s-1}\)\\&
 =c_1+  O(R^{-(s-\frac12)})+O\(\(\frac{\ep}{r}\)^{2s-1}\). 
\end{align*}
Next, define $c_2$ by 
\begin{equation}\label{eq:c2}
c_2 := (n+2s) \int_{0} ^\infty  \left[   \frac{n+2s+2}{ \( 1+q^2 \)^{\frac{n+2s+4}{2}}}- \frac{1}{\(1+q^2\)^{\frac{n+2s+2}{2}}}\right]q^n\,dq.
\end{equation}
 For $|w|<R^\frac12$, we have
\begin{align*}
& \int_0^{\frac{R}{|w|}}
 \left[ \frac{n+2s+2}{ \left(1+q^2 \right)^{\frac{n+2s+4}{2}}}-\frac{1}{ \left(1+q^2 \right)^{\frac{n+2s+2}{2}} }\right]q^n\, dq\\
 &= \int_0^{\infty} 
 \left[\frac{n+2s+2}{ \left(1+q^2 \right)^{\frac{n+2s+4}{2}}} -\frac{1}{ \left(1+q^2 \right)^{\frac{n+2s+2}{2}}}\right]q^n\, dq
 	+O\(R^{-\frac{1+2s}{2}}\) = c_2 + O\(R^{-\frac{1+2s}{2}}\). 
\end{align*}
The previous estimates imply that
\begin{align*}
F&=-\ep^{1-2s} \frac{n+2s}{2}  \int_{S^{n-2}} d \theta\, A(\theta)\int_\R d\xi\, \phi(\xi)\int_{-R^\frac12}^{R^\frac12} dw \, \frac{\phi\(\xi +w\)+\phi\(\xi -w\)-2\phi(\xi) }{|w|^{1+2s}}
\\
&\qquad \times\int_0^{\frac{R}{|w|}} 
 q^{n} \left[   \frac{n+2s+2}{ \left(1+q^2 \right)^{\frac{n+2s+4}{2}}}-\frac{1}{ \left(1+q^2 \right)^{\frac{n+2s+2}{2}}}\right]\,dq 
 +\ep^{1-2s} (o_\ep(1)+ o_{R^{-1}}(1))\\&
 =-\ep^{1-2s} \frac{n+2s}{2}  \int_{S^{n-2}}d \theta\, A(\theta)\int_\R d\xi\, \phi(\xi)\int_{-R^\frac12}^{R^\frac12} dw \, \frac{\phi\(\xi +w\)+\phi\(\xi -w\)-2\phi(\xi) }{|w|^{1+2s}}
\\
&\qquad\times \int_0^{\infty} 
 q^{n} \left[ \frac{n+2s+2}{ \left(1+q^2 \right)^{\frac{n+2s+4}{2}}} -\frac{1}{ \left(1+q^2 \right)^{\frac{n+2s+2}{2}} }\right]\,dq +\ep^{1-2s}(o_\ep(1)+ o_{R^{-1}}(1))\\&
 =\ep^{1-2s}  \left(c_1c_2\int_{S^{n-2}} A(\theta)\,d \theta+o_\ep(1)+o_{R^{-1}}(1) \right).
\end{align*}
We conclude that,
\begin{equation}\label{eq:F-s-final}
\ep^{2s-1}F= c_1c_2\int_{S^{n-2}} A(\theta)\,d \theta+o_\ep(1) +o_{R^{-1}}(1). 
\end{equation}
It remains to check the errors terms.

\medskip
 \noindent
\underline{\bf Step 1a. Estimating $\ep^{2s-1} \int_{\R} \dot{\phi}(\xi) E_1(\xi) \, d \xi$ and $\ep^{2s-1} \int_{\R} \dot{\phi}(\xi) E_2(\xi) \, d \xi$}.
We will show that
\begin{equation}\label{eq:E_1-s-final}
\ep^{2s-1}\int_{\R} \dot{\phi}(\xi)E_1(\xi) \, d \xi= o_\ep(1)
\end{equation}
and
\begin{equation}\label{eq:E_2-s-final}
\ep^{2s-1}\int_{\R} \dot{\phi}(\xi)E_2(\xi) \, d \xi = o_r(1).
\end{equation}
Observe that
\begin{align*}
\int_{0}^{R\ep} \frac{dp}{p^{2s-1}}  \int_{ \frac{r}{p}}^{ \frac{r}{p}} \frac{dt}{ (t^2 +1)^{\frac{n+2s-2}{2}}}
	\leq \int_{0}^{R\ep} \frac{dp}{p^{2s-1}}  \int_{\R} \frac{dt}{ (t^2 +1)^{\frac{n+2s-2}{2}}} = C R^{2-2s} \ep^{2-2s}. 
\end{align*}
First, recalling \eqref{E_1-s}, we estimate
\begin{align*}
|E_1(\xi)|
&\leq \frac{C}{\ep^2}\int_{0}^{R\ep} R^2\ep^2 \frac{dp}{p^{2s-1}} \int_{-\frac{r}{\ep} }^{\frac{r}{\ep} }  \frac{dt}{ (t^2 +1)^{\frac{n+2s-2}{2}}} 
\leq  CR^{4-4s} \ep^{2-2s},
\end{align*}
so that \eqref{eq:E_1-s-final} holds. 
Next, recalling \eqref{E_2-s}, we have
\[
\ep^{2s-1} E_2(\xi) = O(rR^{2-2s}). 
\]
Choosing
\begin{equation}\label{eq:chooseR-gamma}
R = r^{-1},
\end{equation}
we arrive at \eqref{eq:E_2-s-final}.

\medskip
 \noindent
\underline{\bf Step 1b. Estimating $\ep^{2s-1}\int_\R E_3(\xi)\dot\phi(\xi)\,d\xi$ and $E_4$}.
We will show that
\begin{equation}\label{eq:E_3-E_4-s-final}
\ep^{2s-1}\int_\R E_3(\xi)\dot\phi(\xi)\,d\xi= o_\ep(1)
\end{equation}
and 
\begin{equation}\label{eq:E_4-s-final}
 \ep^{2s-1}E_4 = o_\ep(1). 
\end{equation}
First, recalling \eqref{E_3-s}, we write 
\begin{align*}
| E_3(\xi)|
&\leq C\int_{0}^{R\ep}
	\left[ \phi\(\xi + \frac{r}{\ep}\) -  \phi\(\xi - \frac{r}{\ep}\)\right]\frac{1}{ \left( \frac{r^2}{p^2} +1\right)^{\frac{n+2s}{2}}}  \frac{dp}{p^{2s}}\\
&\leq C\int_{0}^{R\ep} 
	\frac{p^{n+2s}}{r^{n+2s}} \frac{dp}{p^{2s}}\\
&= \frac{C}{r^{n+2s}}\int_{0}^{R\ep} p^n \, dp \\
&= \frac{C}{r^{n+2s}} (R\ep)^{n+1}  = o_\ep(1)
\end{align*}
which gives \eqref{eq:E_3-E_4-s-final}. 

Next, by \eqref{eq:asymptotics for phi}, for $|\xi|>2r/\ep$, since $H(\xi \pm\frac{r}{\ep})=H(\xi)$, 
$$\left|\phi\(\xi \pm\frac{r}{\ep}\)-\phi(\xi)\right|\leq \frac{C}{|\xi|^{2s}}.$$
Therefore, recalling  \eqref{E_4-s}, we have
\begin{align*}
|E_4|
&=\bigg|\ep (n+2s) \int_{0}^{R\ep} \frac{dp}{p^{1+2s}} \int_{S^{n-2}} d \theta\, A(\theta)\int_\R d\xi\, \phi(\xi)\\
&\qquad\times \left [ 
\(\phi\(\xi + \frac{r}{\ep}\)-\phi(\xi)\) 
+ \(\phi\(\xi -\frac{r}{\ep}\)-\phi(\xi)\) \right]\frac{ \frac{r}{p}}{ \left( \frac{r^2}{p^2} +1\right)^{\frac{n+2s+2}{2}}} \bigg|\\
&\leq C \ep  \int_{0}^{R\ep} \frac{dp}{p^{1+2s}}  \frac{p^{n+2s+1}}{r^{n+2s+1}}\left[\int_{-\frac{2r}{\ep}}^{\frac{2r}{\ep}}d\xi+\int_{\{|\xi|>\frac{2r}{\ep}\}} \frac{d\xi}{|\xi|^{2s}}\right]\\
&\leq   \frac{C \ep} {r^{n+2s+1}} \int_{0}^{R\ep} p^n\frac{r}{\ep}\,dp
 \leq\frac{C\ep^{n+1}R^{n+1}}{r^{n+2s}}
= o_\ep(1),
\end{align*}
which gives \eqref{eq:E_4-s-final}.

\medskip
 \noindent
\underline{\bf Completion of Step 1.} 
Recall from above that
\begin{align*}
\ep^{2s-1} \int_{\R} \dot{\phi}(\xi) I(\xi)\, d \xi
	&= \ep^{2s-1}\left[F + E_4 +  \int_{\R} \dot{\phi}(\xi) [E_1(\xi) + E_2(\xi) + E_3(\xi)]\, d \xi\right].
\end{align*}
Combining this with \eqref{eq:F-s-final},
\eqref{eq:E_1-s-final},
\eqref{eq:E_2-s-final}, \eqref{eq:chooseR-gamma},  
\eqref{eq:E_3-E_4-s-final} and  \eqref{eq:E_4-s-final}
gives \eqref{eq:I1new-final-s}. 

\medskip
\noindent
\underline{\bf Step 2. Estimating $\ep^{2s-1} \int_{\R} \dot{\phi}(\xi) I_2(\xi) \, d \xi$}. 
 We will show that 
\begin{equation}\label{eq:I3new-final-s}
\ep^{2s-1} \int_{\R} \dot{\phi}(\xi) I_2(\xi) \, d \xi=o_r(1).
\end{equation}
By the monotonicity of $\phi$ and recalling \eqref{bthetatrde},  we have
\begin{align*}
I_2(\xi)&\leq \int_{R\ep}^{r}\frac{dp}{p^{2s+1}}\int_{S^{n-2}}d\theta\int_{-\frac{r}{p}}^{\frac{r}{p}}\left \{ \phi\(\xi + \frac{tp }{\ep}+\frac{Cp^2}{\ep}(1+t^2)\) - \phi\( \xi + \frac{tp}{\ep}\)\right\} \frac{dt}{\left(t^2+1\right)^\frac{n+2s}{2}}\\&
=\int_{R\ep}^{r}\frac{dp}{p^{2s+1}}\int_{S^{n-2}}d\theta\int_{-\frac{r}{p}}^{\frac{r}{p}}\frac{dt}{\left(t^2+1\right)^\frac{n+2s}{2}}
\int_0^1 \dot\phi\(\xi + \frac{tp }{\ep}+\tau\frac{Cp^2}{\ep}(1+t^2)\)\frac{Cp^2}{\ep}(1+t^2)\,d\tau \\&
=\frac{C}{\ep}\int_{R\ep}^{r}\frac{dp}{p^{2s-1}}\int_{S^{n-2}}d\theta\int_{-\frac{r}{p}}^{\frac{r}{p}}\frac{dt}{\left(t^2+1\right)^\frac{n+2s-2}{2}}
\int_0^1 \dot\phi\(\xi + \frac{tp }{\ep}+\tau\frac{Cp^2}{\ep}(1+t^2)\)\,d\tau\\&
=\frac{C}{\ep}\int_{R\ep}^{r}\frac{dp}{p^{2s-1}}\int_{S^{n-2}}d\theta\int_0^1d\tau\\&
\qquad\times \int_{-\frac{r}{p}}^{\frac{r}{p}}\partial_t\left[\phi\(\xi + \frac{tp }{\ep}+\tau\frac{Cp^2}{\ep}(1+t^2)\)\right]\frac{\ep}{p+2\tau Cp^2t}\frac{dt}{\left(t^2+1\right)^\frac{n+2s-2}{2}}.
\end{align*}
Using that $p+2\tau Cp^2t\geq p-2\tau C r p\geq p/2>0$ for $|t|<r/p$ and $r$ small enough,  we integrate by parts to get
\begin{align*}
I_2(\xi)&\leq 
C\int_{R\ep}^{r}\frac{dp}{p^{2s}}\int_{S^{n-2}}d\theta\int_0^1d\tau\int_{-\frac{r}{p}}^{\frac{r}{p}}
\partial_t\left[\phi\(\xi + \frac{tp }{\ep}+\tau\frac{Cp^2}{\ep}(1+t^2)\)\right]\frac{dt}{\left(t^2+1\right)^\frac{n+2s-2}{2}}\\&
=C\int_{R\ep}^{r}\frac{dp}{p^{2s}}\int_{S^{n-2}}d\theta\int_0^1d\tau\left[
\phi\(\xi + \frac{tp }{\ep}+\tau\frac{Cp^2}{\ep}(1+t^2)\)\frac{1}{\left(t^2+1\right)^\frac{n+2s-2}{2}}\Bigg\vert_{t=-\frac{r}{p}}^{t=\frac{r}{p}}\right.\\&
\qquad\left. + (n+2s-2) \int_{-\frac{r}{p}}^{\frac{r}{p}}\phi\(\xi + \frac{tp }{\ep}+\tau\frac{Cp^2}{\ep}(1+t^2)\)\frac{t}{\left(t^2+1\right)^\frac{n+2s}{2}} \, dt\right]\\&
\leq C\int_{R\ep}^{r}\frac{dp}{p^{2s}}\leq C(R\ep)^{-(2s-1)}.
\end{align*}
Similarly, one can prove that
$$I_2(\xi)\geq -C(R\ep)^{-(2s-1)}.$$
We conclude that 
 \begin{equation*}
 \ep^{2s-1} I_2(\xi)=O(R^{-(2s-1)}).
\end{equation*}
Recalling \eqref{eq:chooseR-gamma}, the estimate \eqref{eq:I3new-final-s} follows.  

 \medskip
 
 \noindent
  
 \underline{\bf Conclusion}.
 Recalling \eqref{eq:I-split-s}, we combine \eqref{eq:I1new-final-s} and \eqref{eq:I3new-final-s} to finally obtain
\begin{equation}\label{eq:I-final-s}
\ep^{2s-1}\int_{\R} \dot{\phi}(\xi)G(\xi) \, d\xi
	= c_1c_2\int_{S^{n-2} } A(\theta ) \, d \theta + o_\ep(1)+ o_r(1),
\end{equation}
where $o_\ep(1)$ depends on the parameter $r$. 
From \eqref{eq:aebarI-s} and  \eqref{eq:I-final-s}, we first send $\ep\to0$ and then $r\to0$ to arrive at
\[
\lim_{\ep \to 0} \bar{a}_\ep(x)
	=   c_1c_2 \int_{S^{n-2}} A(\theta) \, d \theta,
\]
uniformly in $Q_\rho$. Recalling \eqref{eq:Atheta-MC}, this gives the desired result with $c_\star = c_1c_2$. 

Lastly, we rewrite $c_2$ in \eqref{eq:c2} to show that $c_\star$ can be written as \eqref{eq:c-star}. 
Since
\[
-\frac{1}{q} \frac{d}{dq} \left[ \frac{1}{\(1+q^2\)^{\frac{n+2s+2}{2}}} \right]
	= \frac{n+2s+2}{\(1+q^2\)^{\frac{n+2s+4}{2}}} \quad \hbox{ for}~q>0,
\]
we can write
\begin{align*}
c_2
	&= (n+2s) \int_{0} ^\infty q^n \left[   - \frac{1}{q} \frac{d}{dq} \left[ \frac{1}{ \( 1+q^2 \)^{\frac{n+2s+2}{2}}}\right]- \frac{1}{\(1+q^2\)^{\frac{n+2s+2}{2}}}\right]\,dq.
\end{align*}
Integrating by parts, we have
\begin{align*}
-\int_{0}^\infty q^{n-1} \frac{d}{dq} \left[ \frac{1}{ \( 1+q^2 \)^{\frac{n+2s+2}{2}}}\right]\,dq
 	&=\int_{0}^\infty (n-1)q^{n-2} \frac{1}{ \( 1+q^2 \)^{\frac{n+2s+2}{2}}}\,dq,
\end{align*}
so that
\begin{align*}
c_2
	&= (n+2s) \int_{0} ^\infty  \left[ (n-1)q^{n-2} \frac{1}{ \( 1+q^2 \)^{\frac{n+2s+2}{2}}}- \frac{q^n}{\(1+q^2\)^{\frac{n+2s+2}{2}}}\right]\,dq\\
	&= (n+2s) \int_{0} ^\infty  \frac{q^{n-2}((n-1)-q^2)} {\(1+q^2\)^{\frac{n+2s+2}{2}}}\,dq.
\end{align*}
Recalling \eqref{finite-energy}, this gives \eqref{eq:c-star} for $c_\star = c_1c_2$. 
\qed

\section{Proof of Theorem \ref{lem:4} for $s \in (0,\frac12)$}\label{sec:s-small}

Throughout this section, we assume that $s \in (0,\frac12)$.

In what follows we denote, as usual, $y=(y',y_n)$ with $y'\in\R^{n-1}$. 
Moreover, we will make  the change of variables $z=Ty$ where 
 $T$ is an orthonormal matrix such that 
\begin{equation}\label{changevar_s<12} \nabla d(x) \cdot (Ty)=y_n.
\end{equation}
We start with some  preliminary results. 
 \begin{lem}\label{lemms<1/2_1}Let  $s \in (0,\frac12)$. Then there exists $C>0$ such that for all  $R,\,\tau>0$, 
 \begin{equation}\label{lemms<1/2_1_est}\int_{\{|y_n|<\tau,\,|y_n|<R|y'|^2\}}\frac{dy}{|y|^{n+2s}}\leq CR^\frac{1+2s}{2}\tau^{\frac{1-2s}{2}}.
 \end{equation}
 \end{lem}
 \begin{proof}
Making a change of variables in $y'$, we compute
\begin{equation*}
\begin{split}
\int_{\{|y_n|<\tau,\,|y_n|<R|y'|^2\}}\frac{dy}{|y|^{n+2s}}
&=\int_{\{|y_n|<\tau\}}\frac{dy_n}{|y_n|^{n+2s}}\int_{\{|y'|>|y_n|^{\frac12}R^{-\frac12}\}}\frac{dy'}{\left(\frac{|y'|^2}{|y_n|^2}+1\right)^\frac{n+2s}{2}}\\&
=\int_{\{|y_n|<\tau\}}\frac{dy_n}{|y_n|^{1+2s}}\int_{\{|z'|>|y_n|^{-\frac12}R^{-\frac12}\}}\frac{dz'}{\left(|z'|^2+1\right)^\frac{n+2s}{2}}\\&
\leq \int_{\{|y_n|<\tau\}}\frac{dy_n}{|y_n|^{1+2s}}\int_{\{|z'|>|y_n|^{-\frac12}R^{-\frac12}\}}\frac{dz'}{|z'|^{n+2s}}\\&
=CR^\frac{1+2s}{2}\int_{\{|y_n|<\tau\}}\frac{dy_n}{|y_n|^\frac{1+2s}{2}}\\&
=CR^\frac{1+2s}{2}\tau^\frac{1-2s}{2}.
\end{split}
\end{equation*}
\end{proof}
   
\begin{lem}\label{lemm2s<1/2_1}
Let  $s \in (0,\frac12)$. There exist $\tau_0,\,R>0$ such that for all $0<\tau\leq \tau_0$, $0\leq\sigma<\tau/2$ and $x\in Q_\rho$,
\begin{align}\label{lemm2s<1/2_est1}\int_{\{d(x+z) > d(x)-\sigma,~-\tau<\nabla d(x) \cdot z < -2\sigma\}}\frac{dz}{|z|^{n+2s}}\leq \int_{ \{|y_n|<\tau,\, |y_n|<R|y'|^2\}}\frac{dy}{|y|^{n+2s}},
\end{align}
and 
\begin{align}\label{lemm2s<1/2_est2}\int_{\{d(x+z) < d(x)+\sigma,~2\sigma<\nabla d(x) \cdot z  <\tau\}}\frac{dz}{|z|^{n+2s}}\leq \int_{ \{|y_n|<\tau,\, |y_n|<R|y'|^2\}}\frac{dy}{|y|^{n+2s}}.
\end{align}
\end{lem}
\begin{proof}
 Since $d\in C^2(Q_{2\rho})$ and Lipschitz continuous in $\R^n$, there exists $C>0$ such that, for all  $x\in Q_\rho$
and $z\in\R^n$,
 $$|d(x+z)-d(x)-\nabla d(x) \cdot z|\leq C|z|^2,$$
 so that 
 \begin{align*}&\{z  : d(x+z) >d(x)-\sigma,~-\tau<\nabla d(x) \cdot z < -2\sigma\}\\&=\{z: d(x+z)-d(x)-\nabla d(x) \cdot z> -\sigma-\nabla d(x) \cdot z>-\frac{\nabla d(x) \cdot z}{2}>\sigma,\,\nabla d(x) \cdot z>-\tau\}\\&
 \subseteq \{z : |\nabla d(x) \cdot z|<\tau,\,|\nabla d(x) \cdot z|< C|z|^2\}.
 \end{align*}
Then, performing the change of variables $z=Ty$ with $T$ as in \eqref{changevar_s<12}, we get
\begin{align*}
\int_{\{d(x+z) > d(x)-\sigma,~-\tau<\nabla d(x) \cdot z < -2\sigma\}}\frac{dz}{|z|^{n+2s}}&\leq \int_{\{  |\nabla d(x) \cdot z|<\tau,\,|\nabla d(x) \cdot z|< C|z|^2\}}\frac{dz}{|z|^{n+2s}}\\&
= \int_{\{ |y_n|<\tau,\,|y_n|< C(|y'|^2+|y_n|^2)\}}\frac{dy}{|y|^{n+2s}}\\&
\leq \int_{\{|y_n|<\tau,\,|y_n|< C(|y'|^2+\tau|y_n|)\}}\frac{dy}{|y|^{n+2s}}.
 \end{align*}
 Let $\tau_0>0$ be such that $1-C\tau_0=\frac12$, then for all $\tau\leq\tau_0$,
 \begin{align*}
 \int_{\{|y_n|<\tau,\, |y_n|< C(|y'|^2+\tau|y_n|)\}}\frac{dy}{|y|^{n+2s}}&\leq  \int_{\{|y_n|<\tau,\, |y_n|< C(|y'|^2+\tau_0|y_n|)\}}\frac{dy}{|y|^{n+2s}}\\&
 = \int_{\{|y_n|<\tau,\, |y_n|< R |y'|^2\}}\frac{dy}{|y|^{n+2s}},
  \end{align*}
  where $R=2C$. This proves \eqref{lemm2s<1/2_est1}. Estimate \eqref{lemm2s<1/2_est2} follows with a similar argument. 
 \end{proof}

The following well-known result, see \cite[Lemma 1]{Imbert09}, is a consequence of Lemmas \ref{lemms<1/2_1}  and \ref{lemm2s<1/2_1}. 
\begin{prop}\label{props<1/2_1}
Let $s \in (0,\frac12)$ and  $x\in Q_\rho$. Then 
   the two following quantities 
 \begin{align*}
\kappa^+[x, d] &=
\nu(\{z : d(x+z) >d(x),~\nabla d(x) \cdot z < 0\}), \\
	\kappa^-[x, d] &= \nu(\{z : d(x+z) < d(x),~\nabla d(x) \cdot z > 0\})
\end{align*}
 are finite. 
 \end{prop}
  Let us proceed with the proof of Theorem \ref{lem:4}.  
  We begin by writing 
 \begin{equation}\begin{split}\label{aepsplits<1/2}
\frac{ a_{\ep}(\xi; x)}{\ep^{2s}}&= \int_{\R^n} \( \phi\(\xi + \frac{d(x+ z)- d(x)}{\ep}\) - \phi\( \xi + \frac{\nabla d(x) \cdot z}{\ep}\) \) \frac{dz}{\abs{z}^{n+2s}}\\&
=\int_{\{ d(x+z) > d(x),~\nabla d(x) \cdot z < 0\}}(\ldots)+\int_{\{d(x+z) < d(x),~\nabla d(x) \cdot z > 0\}}(\ldots)\\&
\quad+\int_{\{d(x+z) > d(x),~\nabla d(x) \cdot z > 0\}}(\ldots)+\int_{\{d(x+z) < d(x),~\nabla d(x) \cdot z < 0\}}(\ldots)\\&
=:I_1(\xi)+I_2(\xi)+I_3(\xi)+I_4(\xi).
 \end{split}
\end{equation}

 \medskip
 
 \noindent
\underline{\bf Step 1: Estimating $\int_\R\dot\phi(\xi) I_1(\xi)\,d\xi$ and $\int_\R\dot \phi(\xi) I_2(\xi)\,d\xi$.} 
We will show that 
\begin{equation}\label{I_1limts<1/2}\lim_{\ep \to 0}\int_\R\dot \phi(\xi) I_1(\xi)\,d\xi=\kappa^+[x, d],
\end{equation} 
and 
\begin{equation}\label{I_2limts<1/2}\lim_{\ep \to 0}\int_\R\dot \phi(\xi) I_2(\xi)\,d\xi=-\kappa^-[x, d].
\end{equation} 
For $\delta>0$, we write
\begin{equation}\begin{split}\label{I_1s<1/2} I_1(\xi)&=\int_{\{ d(x+z) -d(x)>\delta ,~\nabla d(x) \cdot z < -\delta\}}  (\ldots)
\\&\quad+\int_{\{d(x+z) -d(x)>0 ,~-\delta< \nabla d(x) \cdot z < 0\}} (\ldots)
\\&\quad+\int_{\{0<d(x+z) -d(x)<\delta ,~\nabla d(x) \cdot z < -\delta\}} (\ldots)\\&
=:J_1(\xi)+J_2(\xi)+J_3(\xi).
\end{split}
\end{equation}
We first estimate  $\int_\R\dot\phi(\xi) J_1(\xi)\,d\xi$.  By \eqref{eq:asymptotics for phi}, for $|\xi|<\delta/(2\ep)$ and $z\in\R^n$ such that $d(x+z) -d(x)>\delta$ and  $\nabla d(x) \cdot z < -\delta$,
we have
\begin{align*} 
&\phi\(\xi + \frac{d(x+ z)- d(x)}{\ep}\) - \phi\( \xi + \frac{\nabla d(x) \cdot z}{\ep}\)\\&=H\(\xi + \frac{d(x+ z)- d(x)}{\ep}\)-H\( \xi + \frac{\nabla d(x) \cdot z}{\ep}\)
\\&\quad+O\(\left|\xi + \frac{d(x+ z)- d(x)}{\ep}\right|^{-2s}\)+O\(\left|\xi+ \frac{\nabla d(x) \cdot z}{\ep}\right|^{-2s}\)\\&
= 1+O\(\frac{\ep^{2s}}{\delta^{2s}}\). 
\end{align*}
Consequently, by 
\eqref{phi-infyinftyaeptermxiestbis}, 
\begin{align*} 
\int_{-\frac{\delta}{2\ep}}^{\frac{\delta}{2\ep}}\dot\phi(\xi) J_1(\xi)\,d\xi=\int_{\{d(x+z) -d(x)>\delta ,~\nabla d(x) \cdot z < -\delta\}} \frac{dz}{|z|^{n+2s}}+O\(\frac{\ep^{2s}}{\delta^{2s}}\).
\end{align*}
Moreover, by 
\eqref{phi-infyinftyaeptermxiest} and Proposition \ref{props<1/2_1},
\begin{align*} 
\left|\int_{\{|\xi|>\frac{\delta}{2\ep}\}}\dot\phi(\xi) J_1(\xi)\,d\xi\right|\leq 2\int_{\{|\xi|>\frac{\delta}{2\ep}\}}d\xi\,\dot\phi(\xi) 
\int_{\{d(x+z) -d(x)>0 ,~\nabla d(x) \cdot z < 0\}} \frac{dz}{|z|^{n+2s}}=O\(\frac{\ep^{2s}}{\delta^{2s}}\).
\end{align*}
From the last two estimates, we infer that 
\begin{equation}\label{J1_step1_s<1/2}
\int_\R\dot\phi(\xi)J_1(\xi)\,d\xi=\int_{\{d(x+z) -d(x)>0 ,~\nabla d(x) \cdot z < 0\}} \frac{dz}{|z|^{n+2s}}+O\(\frac{\ep^{2s}}{\delta^{2s}}\)+o_\delta(1).
\end{equation}
Next, by \eqref{lemm2s<1/2_est1} with $\sigma=0$ and $\tau=\delta$, for $\delta$ small enough, and  by \eqref{lemms<1/2_1_est}, 
\begin{equation}\label{J2_step1_s<1/2}
\int_\R\dot\phi(\xi) J_2(\xi)\,d\xi=O\(\delta^{\frac{1-2s}{2}}\). 
\end{equation}
Finally, by Proposition \ref{props<1/2_1},
\begin{align*} 
|J_3(\xi)|&\leq 2\int_{\{d(x+z) -d(x)>0 ,~\nabla d(x) \cdot z < 0\}} \one_{\{0< d(x+z) -d(x)<\delta\}}(z) \,\frac{dz}{|z|^{n+2s}}\\&\leq 2 \int_{\{d(x+z) -d(x)>0 ,~\nabla d(x) \cdot z < 0\}}
 \frac{dz}{|z|^{n+2s}}\leq C.
\end{align*}
Since for $x\in Q_\rho$, the set $\{z:d(z)=d(x)\}$ is a smooth surface, we have that 
\begin{equation}\label{chiclosetogamma}\one_{\{0< d(x+z) -d(x)<\delta\}}(z)\to0\quad\text{ a.e. as }\delta\to0.\end{equation}
Therefore, by the Dominated Convergence Theorem, $J_3(\xi)=o_\delta(1)$ and 
\begin{equation}\label{J3_step1_s<1/2}
\int_\R\dot\phi(\xi) J_3(\xi)\,d\xi=o_\delta(1).  
\end{equation}
From \eqref{I_1s<1/2}, \eqref{J1_step1_s<1/2}, \eqref{J2_step1_s<1/2} and \eqref{J3_step1_s<1/2}, letting first $\ep\to0$ and then $\delta\to0$, \eqref{I_1limts<1/2} follows.
The limit in  \eqref{I_2limts<1/2} can be proven with a similar argument. 

 \medskip
 
 \noindent
\underline{\bf Step 2: Estimating $\int_\R\dot\phi(\xi) I_3(\xi)\,d\xi$ and $\int_\R\dot \phi(\xi) I_4(\xi)\,d\xi$.} 
We will show that 
\begin{equation}\label{I_3limts<1/2}\lim_{\ep \to 0}\int_\R\dot \phi(\xi) I_3(\xi)\,d\xi=0,
\end{equation} 
and 
\begin{equation}\label{I_4limts<1/2}\lim_{\ep \to 0}\int_\R\dot\phi(\xi) I_4(\xi)\,d\xi=0.
\end{equation} 
For $\delta>0$, we write
\begin{equation}\begin{split}\label{I_3s<1/2} I_3(\xi)&=\int_{\{z \,:\, d(x+z) -d(x)>\delta ,~\nabla d(x) \cdot z >2\delta\}}  (\ldots)
\\&\quad+\int_{\{d(x+z) -d(x)>0 ,~0< \nabla d(x) \cdot z < 2\delta \}} (\ldots)
\\&\quad+\int_{\{0<d(x+z) -d(x)<\delta ,~\nabla d(x) \cdot z >2\delta\}} (\ldots)\\&
=: J_1(\xi)+J_2(\xi)+J_3(\xi).
\end{split}
\end{equation}
We first estimate  $\int_\R\dot\phi(\xi) J_1(\xi)\,d\xi$.  By \eqref{eq:asymptotics for phi}, for $|\xi|<\delta/(2\ep)$ and $z\in\R^n$ such that $d(x+z) -d(x)>\delta$ and  $\nabla d(x) \cdot z > 2\delta$,
\begin{align*} 
&\phi\(\xi + \frac{d(x+ z)- d(x)}{\ep}\) - \phi\( \xi +\frac{\nabla d(x) \cdot z}{\ep}\)\\&=H\(\xi + \frac{d(x+ z)- d(x)}{\ep}\)-H\( \xi + \frac{\nabla d(x) \cdot z}{\ep}\)
\\&\quad+O\(\left|\xi + \frac{d(x+ z)- d(x)}{\ep}\right|^{-2s}\)+O\(\left|\xi+ \frac{\nabla d(x) \cdot z}{\ep}\right|^{-2s}\)\\&
= O\(\frac{\ep^{2s}}{\delta^{2s}}\). 
\end{align*}
Therefore,  for $|\xi|<\delta/(2\ep)$ and performing the change of variables $z=Ty$ with $T$ as in \eqref{changevar_s<12}, 
\begin{align*} |J_1(\xi)|\leq O\(\frac{\ep^{2s}}{\delta^{2s}}\)\int_{\{\nabla d(x) \cdot z>2\delta\}}\frac{dz}{|z|^{n+2s}}
\leq  O\(\frac{\ep^{2s}}{\delta^{2s}}\)\int_{\{|y|>2\delta\}}\frac{dy}{|y|^{n+2s}}=O\(\frac{\ep^{2s}}{\delta^{4s}}\),
\end{align*}
which, together with \eqref{phi-infyinftyaeptermxiestbis},  implies
$$\int_{-\frac{\delta}{2\ep}}^{\frac{\delta}{2\ep}}\dot\phi(\xi)J_1(\xi)\,d\xi=O\(\frac{\ep^{2s}}{\delta^{4s}}\).
$$
Moreover, for all $\xi\in\R^n$,
\begin{align*}
|J_1(\xi)|\leq \int_{\{|y|>2\delta\}}\frac{2}{|y|^{n+2s}}\,dy=O\(\frac{1}{\delta^{2s}}\),
\end{align*}
so that, by 
\eqref{phi-infyinftyaeptermxiest}, 
\begin{align*} 
\int_{\{|\xi|>\frac{\delta}{2\ep}\}}\dot\phi(\xi)J_1(\xi)\,d\xi=O\(\frac{\ep^{2s}}{\delta^{4s}}\).
\end{align*}
We infer that 
\begin{equation}\label{J1_step2_s<1/2}
\int_\R\dot\phi(\xi)J_1(\xi)\,d\xi=O\(\frac{\ep^{2s}}{\delta^{4s}}\).
\end{equation}
Next, let us estimate   $\int_\R\dot\phi(\xi) J_2(\xi)\,d\xi$. By the monotonicity of $\phi$, making the change of variables $z=Ty$ with $T$ as in \eqref{changevar_s<12},
and then taking $p=|y'|$, $t=y_n/p$, 
 we estimate
 \begin{align*}
 J_2(\xi)&\leq \int_{\{0<\nabla d(x) \cdot z <2\delta\}} 
 \( \phi\(\xi +  \frac{\nabla d(x) \cdot z+C|z|^2}{\ep}\) - \phi\( \xi + \frac{\nabla d(x) \cdot z}{\ep}\) \) \frac{dz}{\abs{z}^{n+2s}}\\&
 =\int_{\{0<y_n <2\delta\}} 
 \( \phi\(\xi +  \frac{y_n+C(|y'|^2+y_n^2)}{\ep}\) - \phi\( \xi +\frac{ y_n}{\ep}\) \) \frac{dy}{\abs{y}^{n+2s}}\\&
 =\mathcal{H}^{n-2}(S^{n-2})\int_0^\infty \frac{dp}{p^{1+2s}}\int_0^{\frac{2\delta}{p}}\( \phi\(\xi +  \frac{tp+Cp^2(1+t^2)}{\ep}\) - \phi\( \xi +\frac{ tp}{\ep}\) \) \frac{dt}{(1+t^2)^\frac{n+2s}{2}}\\&
 =\int_0^r \frac{dp}{p^{1+2s}}\,(\ldots)+\int_r^\infty \frac{dp}{p^{1+2s}}\,(\ldots)\\&
 =: J_2^1(\xi)+J_2^2(\xi),
 \end{align*}
 with $r\geq\delta$ to be determined. For the first term above, we have
 \begin{align*}
 &J_2^1(\xi)\\
 &=\frac{C}{\ep}\int_0^r dp\,p^{1-2s} \int_0^{\frac{2\delta}{p}} \frac{dt}{(1+t^2)^\frac{n+2s-2}{2}}\int_0^1\dot\phi\(\xi +  \frac{tp+\tau Cp^2(1+t^2)}{\ep}\) d\tau\\&
 =\frac{C}{\ep}\int_0^r dp\,p^{1-2s} \int_0^{\frac{2\delta}{p}} \frac{dt}{(1+t^2)^\frac{n+2s-2}{2}}\int_0^1\partial_t \left[\phi\(\xi +  \frac{tp+\tau Cp^2(1+t^2)}{\ep}\)\right]
 \frac{\ep}{p(1+2t\tau Cp)} d\tau\\&
 \leq C\int_0^r \frac{dp}{p^{2s}} \int_0^1d\tau \int_0^{\frac{2\delta}{p}}\partial_t \left[\phi\(\xi +  \frac{tp+\tau Cp^2(1+t^2)}{\ep}\)\right]dt,
  \end{align*}
 where we used that $p(1+2t\tau Cp)\geq p/2$ if $0\leq tp\leq 2\delta$ and $\delta$ is small enough. 
 Integrating with respect to $t$, we obtain
 \begin{align*}
 J_2^1(\xi)&\leq C\int_0^r \frac{dp}{p^{2s}} \int_0^1\left[\phi\(\xi +  \frac{2\delta+\tau C(p^2+4\delta^2)}{\ep}\) -\phi\(\xi +\frac{\tau Cp^2}{\ep} \)\right] d\tau\\&
 \leq  C\int_0^r \frac{dp}{p^{2s}}=Cr^{1-2s}. 
 \end{align*}
We also estimate
  \begin{align*}
  J_2^2(\xi)\leq C\int_r^\infty \frac{dp}{p^{1+2s}}\int_0^{\frac{2\delta}{p}}dt=\frac{C\delta}{r^{1+2s}}.
  \end{align*}
 Choosing $r=r(\delta)$ such that $r=o_\delta(1)$ and $\delta/r^{1+2s}=o_\delta(1)$, we obtain $ J_2(\xi)\leq o_\delta(1).$ The lower bound can be proven with a similar argument. We conclude that 
 \begin{equation}\label{J2_step2_s<1/2}
\int_\R\dot\phi(\xi)J_2(\xi)\,d\xi=o_\delta(1). 
\end{equation}
Finally, for $\tau_0$ as in Lemma  \ref{lemm2s<1/2_1} and $\delta<\tau_0/2$, from \eqref{lemm2s<1/2_est2} and \eqref{lemms<1/2_1_est},
\begin{align*}
|J_3(\xi)|&\leq 2 \int_{\{0<d(x+z) -d(x)<\delta ,~\nabla d(x) \cdot z >2\delta\}} \frac{dz}{|z|^{n+2s}}\\&
=2 \int_{\{d(x+z) -d(x)<\delta,~\nabla d(x) \cdot z>2\delta \}}\one_{\{0<d(x+z) -d(x)<\delta\}}(z) \,\frac{dz}{|z|^{n+2s}}\\&
\leq 2 \int_{\{d(x+z) -d(x)<\delta,~2\delta<\nabla d(x) \cdot z< \tau_0\}} \frac{dz}{|z|^{n+2s}}+ 2 \int_{\{\nabla d(x) \cdot z>\tau_0\}} \frac{dz}{|z|^{n+2s}}\\&
\leq C. 
 \end{align*}
Recalling \eqref{chiclosetogamma}, we see that by the Dominated Convergence Theorem, $J_3(\xi)= o_\delta(1)$ and 
\begin{equation}\label{J3_step2_s<1/2}
\int_\R\dot\phi(\xi) J_3(\xi)\,d\xi=o_\delta(1).  
\end{equation}

From \eqref{I_3s<1/2}, \eqref{J1_step2_s<1/2}, \eqref{J2_step2_s<1/2} and \eqref{J3_step2_s<1/2}, letting first $\ep\to0$ and then $\delta\to0$, \eqref{I_3limts<1/2} follows.
The limit in  \eqref{I_2limts<1/2} can be proven with a similar argument.

 \medskip
 
 \noindent
 \underline{\bf Conclusion}.
Recalling \eqref{aepsplits<1/2}, we combine \eqref{I_1limts<1/2},  \eqref{I_2limts<1/2},  \eqref{I_3limts<1/2}, and  \eqref{I_4limts<1/2}   to complete the proof.
\qed
 
 \section*{Acknowledgements}

The first author has been supported by the NSF Grant DMS-2155156 ``Nonlinear PDE methods in the study of interphases.'' 
The second author has been supported by the Australian Laureate Fellowship FL190100081 ``Minimal surfaces, free boundaries and partial differential equations.''
Both authors acknowledge the support of NSF Grant DMS RTG 18403.



\end{document}